\documentclass{amsart}

\usepackage{amssymb}
\usepackage[bookmarks=true, bookmarksopen=true,%
bookmarksdepth=3,bookmarksopenlevel=2,%
colorlinks=true,%
linkcolor=blue,%
citecolor=blue,%
filecolor=blue,%
menucolor=blue,%
urlcolor=blue]{hyperref}
\usepackage{cleveref}
\usepackage[all]{xy}

\usepackage{tikz}
\usepackage{tikz-cd}
\usepackage{mathtools}
\usetikzlibrary{arrows,decorations.markings, matrix}
\usepackage{float}
\usetikzlibrary{decorations.pathreplacing}
\usepackage[textsize=scriptsize]{todonotes}
\usepackage{verbatim}

\theoremstyle{definition}
\newtheorem{nul}{}[section]
\newtheorem{dfn}[nul]{Definition}

\newtheorem{rmk}[nul]{Remark}

\newtheorem{exm}[nul]{Example}

\newtheorem*{qst*}{Question}
\newtheorem*{dfn*}{Definition}
\newtheorem*{axm*}{Axiom}
\newtheorem*{ntn*}{Notation}
\newtheorem*{exm*}{Example}
\newtheorem*{exr*}{Exercise}
\newtheorem*{int*}{Intuition}
\newtheorem*{rmk*}{Remark}

\newtheorem{thm}[nul]{Theorem}
\newtheorem{prop}[nul]{Proposition}
\newtheorem{lem}[nul]{Lemma}

\newtheorem{cor}[nul]{Corollary}
\newtheorem*{thm*}{Theorem}
\newtheorem*{prop*}{Proposition}
\newtheorem*{cor*}{Corollary}
\newtheorem*{lem*}{Lemma}
\newtheorem*{cnj*}{Conjecture}

\newcommand{\bF}{\mathbf{F}}

\newcommand{\bS}{\mathbf{S}}
\newcommand{\bZ}{\mathbf{Z}}

\newcommand{\cO}{\mathcal{O}}

\newcommand{\Mod}{\mathrm{Mod}}
\newcommand{\LMod}{\mathrm{LMod}}
\newcommand{\RMod}{\mathrm{RMod}}
\newcommand{\Bimod}{\mathrm{Bimod}}
\newcommand{\LModft}[1]{\mathrm{LMod}{(#1)^{\mathrm{ft}}}}

\newcommand{\Modomega}[1]{\mathrm{Mod}{(#1)^{\omega}}}
\newcommand{\LModomega}[1]{\mathrm{LMod}{(#1)^{\omega}}}

\newcommand{\RModft}{\mathrm{RMod}^{\mathrm{ft}}}

\newcommand{\Maps}{\mathrm{Maps}}
\newcommand{\Hom}{\mathrm{Hom}}
\newcommand{\Fun}{\mathrm{Fun}}
\newcommand{\op}{\mathrm{op}}
\newcommand{\PrL}{\mathrm{Pr}^{\mathrm{L}}}

\usepackage{color}

\newcommand{\F}{\mathbf{F}}

\newcommand{\Q}{\mathbf{Q}}
\newcommand{\Z}{\mathbf{Z}}

\newcommand{\mrm}[1]{\mathrm{#1}}

\DeclareMathOperator{\SL}{SL}

\DeclareMathOperator{\Ind}{Ind}

\DeclareMathOperator{\Nm}{Nm}
\DeclareMathOperator{\Spec}{Spec\,}

\DeclareMathOperator{\Res}{Res}

\DeclareMathOperator{\Proj}{Proj\,}

\DeclareMathOperator{\pt}{pt}

\DeclareMathOperator{\Cat}{Cat}

\DeclareMathOperator{\Burn}{\mathrm{Burn}}

\renewcommand\a\alpha
\renewcommand\b\beta
\newcommand\g\gamma
\renewcommand\d\delta

\newcommand\cC{\mathcal{C}}
\newcommand\cD{\mathcal{D}}

\newcommand\cS{\mathcal{S}}

\DeclareMathOperator{\KU}{KU}

\newcommand{\bC}{\mathbf{C}}
\newcommand{\bQ}{\mathbf{Q}}

\newcommand{\PSL}{\mathrm{PSL}}

\newcommand{\ex}{\mathrm{ex}}

\numberwithin{equation}{subsection}
\numberwithin{nul}{subsection}
\begin{document}
\title{$G$-spectra of cyclic defect}
\author{Tony Feng, David Treumann, and Allen Yuan}

\maketitle

\begin{abstract}Brou\'{e}'s Abelian Defect Conjecture predicts interesting derived equivalences between derived categories of modular representations of finite groups. We investigate a generalization of Brou\'{e}'s Conjecture to ring spectrum coefficients and prove this generalization in the cyclic defect case, following an argument of Rouquier. 
\end{abstract}

\tableofcontents

\section{Introduction}\label{sec:pre-intro}

Representation theory behaves significantly differently depending on whether the ground field has characteristic $0$ or characteristic $p$. Since more general rings, like the $p$-adic integers $\bZ_p$, in some sense interpolate between characteristics $0$ and $p$, the study of representations on modules over such rings plays an important role in modular representation theory.

In this paper, we investigate representations over an even more general class of rings that arises in algebraic topology, called \emph{ring spectra}. Ordinary commutative rings are ring spectra, but ring spectra also include differential graded algebras, as well as non-algebraic objects such as (to name two of many) the one that represents complex $K$-theory and the one that represents the stable homotopy groups of spheres. In a way, enlarging the class of rings makes the ``interpolation between characteristics $0$ and $p$'' more finely resolved. There is a rich theory and we hope that the tools and concepts of that theory could be applied to make progress on representation-theoretic questions, even on classical problems that have nothing to do with algebraic topology. We review some of this theory, and some of its precedent applications to other areas, in \S\ref{intro:spectra}.

Conversely, some of those representation-theoretic questions suggest very natural questions in algebraic topology that have seemingly not been considered before. In this paper, we attend to one such question, Brou\'e's Abelian Defect Conjecture \cite{Br90}. It predicts that the derived categories of mod $p$ or $p$-adic representations of a finite group $G$ and of a normalizer subgroup $N_G(D)$ have direct factors in common when $D$ is (1) abelian of $p$-power order and (2) a defect group of one of the mod $p$ blocks of $G$. We review it in greater detail in \S\ref{subsec:romrt}--\ref{subsec:tcogm}.

For thirty years Brou\'e's Conjecture has been a guiding problem in modular representation theory. The purpose of this paper is to formulate the analogous problem with $p$-adic rings replaced by $p$-complete ring spectra: do the categories of $p$-complete $G$-spectra and of $p$-complete $N_G(D)$-spectra have direct factors in common? 

In \S\ref{subsec:impliesbroue} we show that the original Brou\'e conjecture is a consequence of our spectrum version. We hope that the original conjecture could be attacked by applying some of the tools and concepts of stable homotopy theory: the Segal Conjecture/Carlsson's Theorem, Tate fixed points/Frobenius, organizing by chromatic level, etc.; we discuss these further in  \S\ref{intro:spectra}, \S\ref{subsec:Knlocalalg}--\ref{subsec:Knlocalfun}, \S\ref{sec:strong-segal}. However, this hope is not realized in this paper. Before pursuing it very seriously, it is natural to first ask if there is any evidence that the spectrum version is true.

In \S\ref{sec:rouquier}, we prove our spectrum version when $D$ is cyclic. The result does not illuminate the original form of the Brou\'e conjecture, which has been known for a long time in the cyclic defect case. Different proofs have been given by Rickard and by Rouquier. Both arguments end in the construction of ``tilting complexes'' of $(G,N_G(D))$-bimodules that implement the derived equivalence. In general there are many obstructions to lifting a tilting complex to a $G \times N_G(D)^{\op}$-spectrum. We deduce the desired equivalence of categories by showing that Rouquier's tilting complex does lift.

\subsection*{Acknowledgements}
The authors would like to thank Robert Burklund and Jay Taylor for helpful conversations. The first author was supported by NSF grant DMS 1902927. The third author was supported in part by NSF grant DMS-2002029.

\section{Brou\'{e}'s Conjecture over ring spectra}\label{sec:intro}

\subsection{Review of modular representation theory}
\label{subsec:romrt}
This is a paper about finite group actions on $p$-complete spectra, but we will start this extended introduction by reviewing an example and some of the theory of finite group actions on $p$-adic abelian groups. Three basic questions in modular representation theory are: 
\begin{enumerate}
\item What are the simple representations over a $p$-adic field of characteristic $0$? 
\item What are the simple representations over a finite field of characteristic $p$?
\item What are the projective representations in characteristic $p$, or over a $p$-adic ring of integers?
\end{enumerate}
The answer to these questions, and some information about blocks, defects, and Brauer trees, are given below in the case $G = \PSL_2(\bF_7)$ and $p = 7$. We will use this example to illustrate some of the concepts of modular representation theory. 

Conventionally, one chooses the coefficient rings and fields to have sufficiently many roots of unity. We wish to avoid adding $p$th roots of unity (or $4$th roots when $p = 2$) to our coefficients for algebraic topology reasons we'll come to later \S\ref{intro:spectra}. We describe the representation theory of $\PSL_2(\bF_7)$ over $\bZ_7$ and $\bQ_7$.

\subsubsection{Irreducible $\bQ_7[\PSL_2(\bF_7)]$-modules}
Let $G = \PSL_2(\bF_7)$ be the simple group of order 168 and let $p = 7$. There are 5 irreducible representations of $G$ on $\bQ_7$-vector spaces that we will call 
\[
\mathbf{1}, \mathbf{7},\mathbf{8},\mathbf{6}\text{ and }\mathbf{3}\overline{\mathbf{3}} \qquad \text{of dimensions $1,7,8,6,$ and $6$ respectively.}
\]
The trivial representation is $\mathbf{1}$; the other representations have names coming from the theory of $\SL_2$: $\mathbf{7}$ is the Steinberg representation, $\mathbf{8}$ is in the principal series, the two $6$-dimensional representations are in the discrete series.

The representations $\mathbf{1},\mathbf{7},\mathbf{8}$, and $\mathbf{6}$ are all absolutely irreducible but if we adjoin $\sqrt{-7}$ or $\sqrt[7]{1}$ to $\bQ_7$, the representation $\mathbf{3}\overline{\mathbf{3}}$ splits as a sum of two $3$-dimensional representations.

\subsubsection{Simple and indecomposable projective $\bZ_7[\PSL_2(\bF_7)]$-modules}
The theory of projective covers gives a bijection between the simple and the indecomposable projective $\bZ_p[G]$-modules. Over $\bZ_7$, there are four indecomposable projective $\PSL_2(\bF_7)$-representations: a lattice in $\mathbf{7}$, a lattice in $\mathbf{1} \oplus \mathbf{6}$, a lattice in $\mathbf{6} \oplus \mathbf{8}$, and a lattice in $\mathbf{8} \oplus \mathbf{3}\overline{\mathbf{3}}$.

The simple quotients of these modules are $\bF_7$-vector spaces of dimensions $7$, $1$, $5$, and $3$, respectively; they don't play a prominent role in this section, and in general simple modules don't play a prominent role in this paper.

\subsubsection{Blocks}
If $G$ is any finite group and $p$ is any prime, then the group algebra $\bZ_p[G]$ is a product of associative rings called \emph{block algebras}:
\begin{equation}
\label{eq:timestimes}
\bZ_p[G] = \bZ_p[G]b_1 \times \bZ_p[G]b_2 \times \cdots
\end{equation}
The way direct products of rings work, every $\bZ_p[G]$-module $M$ is canonically a direct sum of modules for the block algebras
\[
M = Mb_1 \oplus Mb_2 \oplus \cdots
\]
The notation $Mb_i$ means $M$ is a module for $\bZ_p[G]b_i$. We say that ``$M$ belongs to the block $b_i$'' if $M = Mb_i$.

Any indecomposable $\bZ_p[G]$-module belongs to a single block. All the $\bZ_p[G]$-lattices in an irreducible $\bQ_p[G]$-module belong to the same block. Thus, the blocks partition the set of irreducible $\bQ_p[G]$-modules. One way to name or classify the blocks of $G$ is by describing this partition. For example, there are two blocks of $\bZ_7[\PSL_2(\bF_7)]$ and they are labeled by 
\begin{equation}
\label{eq:18633-7}
\{\mathbf{1},\mathbf{8},\mathbf{6},\mathbf{3}\overline{\mathbf{3}}\} \text{ and } \{\mathbf{7}\}
\end{equation}
In general the block containing the trivial representation is called the \emph{principal block}. When $G = \PSL_2(\bF_7)$, we can call the other block the \emph{Steinberg block}, since the only finitely generated indecomposable module that belongs to it is a lattice in the Steinberg representation $\mathbf{7}$.

\subsubsection{Defect groups and a Brauer tree} ``Defects'' are invariants of blocks.

A block $b$ has \emph{defect zero} if there's only one irreducible $\bQ_p[G]$-module that belongs to $b$ and every lattice in that irreducible module is projective. For instance, the Steinberg block has defect zero. A block of defect zero is categorically and homologically boring --- its category of modules is the same as the category of modules over $\bZ_p$, or of a division ring over $\bZ_p$. 

Groups whose $p$-Sylow is of order $p$ furnish the basic examples of blocks of \emph{defect one}. The principal block of such a group has defect one. For instance the principal block of $\PSL_2(\bF_7)$ (labeled $\{\mathbf{1},\mathbf{8},\mathbf{6},\mathbf{3}\overline{\mathbf{3}}\}$ in \eqref{eq:18633-7}) is a block of defect one.

Brauer observed a ``tree'' structure in the decomposition matrix of a block of defect one. The decomposition matrix records how the indecomposable projective $\bZ_p[G]$-modules break up into irreducible $\bQ_p[G]$-modules. Brauer showed that, sticking to modules in a block of defect one, this matrix is the incidence matrix of a tree. 

For example, the Brauer tree of the principal block of $\PSL_2(\bF_7)$ is
\begin{equation}
\label{eq:treeG}
\xymatrix{
\mathbf{1} \ar@{-}[r] & \mathbf{6} \ar@{-}[r] & \mathbf{8} \ar@{-}[r] & \mathbf{3} \overline{\mathbf{3}}
}
\end{equation}
Its vertices are labeled by the irreducible $\bQ_7[\PSL_2(\bF_7)]$-modules in the principal block. It is natural to label its edges by the indecomposable projective $\bZ_7[\PSL_2(\bF_7)]$-modules: if $P \otimes \bQ_7$ splits up as $M \oplus N$, draw an edge between $M$ and $N$.
\medskip

In general the ``defect'' of a block is more detailed information than a number $n$; it is a conjugacy class of subgroups of $G$ of order $p^n$. A simple way to characterize the defect group is as follows: the \emph{defect group} of $b$ is the smallest $p$-subgroup $Q \subset G$ such that every module $M$ belonging to $b$ is a direct summand of $\Ind_Q^G(\text{a $\bZ_p[Q]$-module})$. 

\subsubsection{Coefficients}\label{sssec: coeffs}
In our running example of $\PSL_2(\bF_7)$, Brou\'e's conjecture (to be introduced below) is easy to check over $\bZ_7$. But in general we will need to consider a larger coefficient system.

Let $\cO$ be a finite extension of $\bZ_p$ whose fraction field $K$ has characteristic $0$ and whose residue field $k$ has characteristic $p$. Sometimes modular representation theorists call the triple $(K,\cO,k)$ a ``$p$-modular system.'' We will take a triple of the following form: 
\begin{itemize}
\item $k = \F_q$ contains a primitive $n$th root of unity whenever $G$ has an element of order $n$ which is prime to $p$. 
\item $\cO = \Z_q$ is the ring of Witt vectors of $k$. (This is the minimal extension of $\Z_p$ with residue field $\F_q$.) 
\item $K = \Q_q$ is the fraction field of $\cO$. 
\end{itemize}
\begin{rmk}
\label{rem:assez-gros}
The reduction map $\cO[G] \to k[G]$ always induces a bijection on blocks. The hypothesis on $\bF_q$ has the following additional consequence for blocks: that the map $\bF_q[G] \to k'[G]$ induces a bijection on blocks for any extension field $k'/\bF_q$, and the map $\bZ_q[G] \to \cO'[G]$ induces a bijection on blocks for any extension $\cO'$ of $\bZ_q$.
\end{rmk}

\begin{rmk}
A standard additional hypothesis on $p$-modular systems is that $\cO$ and $K$ contain $\exp(G)$th roots of unity, where $\exp(G)$ is the exponent of $G$ --- sometimes $(K,\cO,k)$ is called a ``splitting $p$-modular system'' when this hypothesis holds. We warn that $(\bQ_q,\bZ_q,\bF_q)$ is usually \emph{not} a splitting $p$-modular system in this sense, for example not if $p$ is odd and divides the order of $|G|$, or if $p = 2$ and the $2$-Sylow subgroup of $G$ is not elementary abelian. Our reasons for avoiding $p$th roots of unity are explained in \S\ref{intro:spectra}.

Brou\'e's conjecture, introduced below, concerns blocks of $\cO[G]$ and $k[G]$ where $(K,\cO,k)$ is a ``sufficiently large'' $p$-modular system. Usually, ``sufficiently large'' at least implicitly means ``splitting'' but we will work with $(K,\cO,k) = (\bQ_q,\bZ_q,\bF_q)$, where $\bF_q$ obeys the condition above. Though not splitting, the conjecture for $(\bQ_q,\bZ_q,\bF_q)$ implies the conjecture for any larger $(K,\cO,k)$, and conversely Rickard's refined form of the Brou\'e conjecture for $\bF_q$ \cite{Rickard-splendid} implies it for $\bZ_q$ \cite[\S 5]{Rickard-splendid}.\end{rmk}

\subsection{Brou\'e's conjecture}
\label{intro:6}

Write $D^b(\Z_q[G]b)^{\mrm{fg}}$ for the bounded derived category of finitely generated $\Z_q[G]b$-modules. Whenever the defect group $D$ of $b$ is abelian, Brou\'e's conjecture predicts an equivalence of derived categories
\begin{equation}
\label{eq:broue-db}
D^b(\Z_q[G]b)^{\mrm{fg}} \cong D^b(\Z_q[N_G(D)]b')^{\mrm{fg}}
\end{equation}
for some block $b'$ of the normalizer $N_G(D)$ of $D$. Given $b$, there is an explicit recipe to determine $b'$, which is called the ``Brauer correspondent'' of $b$. In particular, if $b$ is the principal block of $\Z_q[G]$ then $b'$ is the principal block of $\Z_q[N_G(D)]$.

The conjecture has been known for a long time for blocks of cyclic defect --- proved first by Rickard \cite{rickard} and later by Rouquier \cite{rouquier-cyclic}. In case $G = \PSL_2(\bF_7)$, the ring $\Z_q$ is $\bZ_7$ itself; the example is pretty typical and Rouquier's equivalence can be described quickly, as follows. 

The normalizer of the defect group of the principal block of $G = \PSL_2(\bF_7)$ is the ``Borel'' subgroup $B \subset G$, of order $21$. The group algebra $\bZ_7[B]$ has only one block. This block has defect one and its Brauer tree is
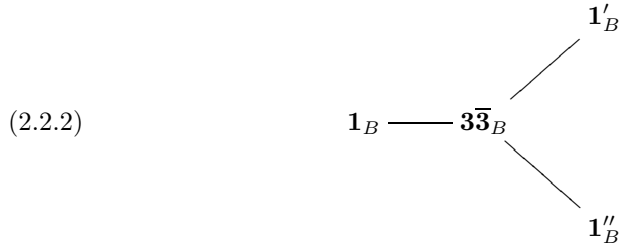
\begin{equation}
\label{eq:treeB}
\begin{gathered}
\xymatrix{
& & \mathbf{1}'_B \\
\mathbf{1}_B \ar@{-}[r] & \mathbf{3}\overline{\mathbf{3}}_B \ar@{-}[rd] \ar@{-}[ru]\\
& & \mathbf{1}''_B
}
\end{gathered}
\end{equation}
The central vertex is another module that would split over a field with $\sqrt{-7}$, decorated with a subscript $B$ to distinguish it from the irreducible $\bQ_7[G]$-module with the same property and similar name. The other vertices are $1$-dimensional modules, two of them nontrivial. The Rouquier equivalence sends a $G$-module $M$ in the principal block to the two-term complex
\begin{equation}
\label{eq:2-term-complex}
\Res^G_B(M) \to Q \otimes \Hom_G(P,M)
\end{equation}
In the formula \eqref{eq:2-term-complex}, $P$ and $Q$ are indecomposable projectives, which we can specify by indicating their corresponding edges in the Brauer trees:
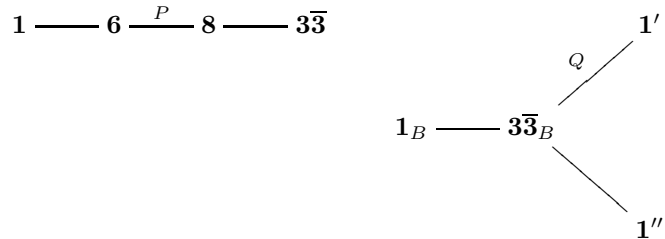
\[
\xymatrix{
\mathbf{1} \ar@{-}[r] & \mathbf{6} \ar@{-}[r]^P & \mathbf{8} \ar@{-}[r] & \mathbf{3} \overline{\mathbf{3}}}
\qquad 
\xymatrix{
& & \mathbf{1}' \\
\mathbf{1}_B \ar@{-}[r] & \mathbf{3}\overline{\mathbf{3}}_B \ar@{-}[rd] \ar@{-}[ru]^Q\\
& & \mathbf{1}''
}
\]
The differential in \eqref{eq:2-term-complex} is a natural transformation $\Res^G_B(-) \to Q \otimes \Hom_G(P,-)$. A Yoneda/Morita argument identifies the set of such natural transformations with a ball in a $p$-adic vector space, specifically with $\Hom_B(\Res^G_B P,Q) \cong \bZ_7^5$. A little care is necessary in choosing the differential in this ball: the corresponding homomorphism $\Res^G_B P \to Q$ must be surjective. This condition is both closed and open in the usual $7$-adic metric on $\bZ_7^5$.

\subsection{Triangulated categories of $G$-modules}
\label{subsec:tcogm}
The block decomposition \eqref{eq:timestimes} of $\Z_q[G]$ induces a decomposition of the derived category $D^b(\Z_q[G])^{\mathrm{fg}}$ as a direct product of triangulated categories
\[
D^b(\Z_q[G])^{\mrm{fg}} \cong D^b(\Z_q[G]b_1)^{\mrm{fg}} \times D^b(\Z_q[G]b_2)^{\mrm{fg}} \times \cdots
\]
Each of these categories has a natural dg enrichment. In other words, each of them is the homotopy category of a $\Z_q$-linear stable $\infty$-category. Our notation for these $\infty$-categories is not the usual one in representation theory: we write $\LModft{\Z_q[G]b}$ for the $\Z_q$-linear stable $\infty$-category whose homotopy category is $D^b(\Z_q[G]b)^{\mrm{fg}}$. Let us make some comments about this notation, which is developed in \S\ref{subsec:finite-k}--\ref{subsec:fin-left-mod}.

\begin{itemize}
\item $\LMod$ stands for ``left module spectra'' --- $\LMod(R)$ is the $\infty$-category of left module spectra over the ring spectrum $R$. Every ring in the usual sense --- i.e., every ``discrete ring'' --- determines a ring spectrum, namely its \emph{Eilenberg-MacLane spectrum}. In this paper we abuse notation and use the same symbol for a discrete ring as for its Eilenberg-MacLane spectrum. Thus, $\LMod(\Z_q[G]b)$ is the $\infty$-category of left module spectra over the Eilenberg-MacLane spectrum of $\Z_q[G]b$. 
\item The homotopy category of $\LMod(\Z_q[G]b)$ is equivalent not to $D^b(\bZ_q[G]b)^{\mathrm{fg}}$ but to $D(\Z_q[G]b)$, the unbounded derived category of $\bZ_q[G]b$. Its objects are complexes that are not required to be bounded, or to obey any other finiteness condition.

\item The superscript in $\LModft{\bZ_q[G]b}$ stands for ``finite type.'' The notation is based on the following characterization of $D^b(\Z_q[G]b)^{\mrm{fg}}$ as a subcategory of $D(\Z_q[G]b)$: it is the full subcategory spanned by those complexes whose underlying complex of $\Z_q$-modules is bounded and finitely generated in each degree.

\end{itemize}

\subsection{Spectra}\label{intro:spectra}

Extraordinary cohomology theories (cobordism, $K$-theory, \dots) are represented by spectra.   The category of spectra can be be viewed as a homotopy theoretic refinement of the derived category of abelian groups.  For instance, it is triangulated, and the category of abelian groups embeds inside spectra by the Eilenberg-MacLane spectrum construction $A \mapsto H^*(-;A)$.  Any spectrum $X$ has homotopy groups, which are the graded abelian group
\[
\pi_* X = H^{-*}(\mathrm{pt}; X).
\]
Moreover, the category of spectra has a symmetric monoidal structure, and so one can consider ring spectra and commutative ring spectra, analogously to DGA's and CDGA's in $D(\Z)$. 

In this ``higher algebra'' of ring spectra, the initial ring spectrum is $\bS$, the sphere spectrum, whose homotopy groups are the stable homotopy groups of spheres.  The $\infty$-category of spectra is equivalent $\LMod(\bS)$ --- for a commutative ring spectrum such as $\bS$, we usually drop the $\mathrm{L}$ and write $\Mod(\bS) := \LMod(\bS)$.  
In the body of the paper we will assume a greater familiarity with spectra and with $\infty$-categories; here in the introduction we will make some basic comments:
\begin{enumerate}
\item The map $g: \bS \to \bZ$ induces an isomorphism in homotopy groups in non-positive degrees, and the positive degree homotopy groups of $\bS$ are the well-studied \emph{stable homotopy groups of spheres}. These groups form a graded commutative ring under composition, elements in positive degree are known to be torsion \cite{serre} and nilpotent \cite{nishida}. 

\item For each prime $p$ there is a natural $p$-completion of $\bS$, another commutative ring spectrum denoted $\bS_p$. Its homotopy groups are the homotopy groups of $\bS$, tensored with $\bZ_p$.
\item When $n$ is prime to $p$, there is a natural construction that ``adjoins $n$th roots of unity to $\bS_p$'', e.g. \cite[Ex. 5.2.7]{elliptic-cohomology}. We will denote the result by $\bS_q$, where $q$ is the cardinality of the field obtained by adjoining $n$th roots of unity to $\bF_p$. Its homotopy groups are the homotopy groups of $\bS$, tensored with $\bZ_p(\sqrt[n]{1})$.  On the other hand, when $p$ is odd, it is not possible to adjoin $p$th roots of unity to $\bS_p$  \cite[\S A.6.iii]{lawson}, and for $p=2$, it is not possible to adjoin $4$th roots of unity to $\bS_2$ \cite{SVW}.
\end{enumerate}

\subsubsection{Relation to classical algebra}
Regarding $\bZ$ as a ring spectrum, $\Mod(\bZ)$ is a stable $\infty$-category whose homotopy category is $D(\bZ)$. Algebra over $\bS$ and over $\bZ$ is related by the unique homotopy class of ring spectrum maps $g: \bS \to \bZ$ and the induced extension and restriction of scalar functors 
\[
g^*: \Mod(\bS) \to \Mod(\bZ) \qquad g_*:\Mod(\bZ) \to \Mod(\bS)
\]
The map $g: \bS \to \bZ$ can be thought of as a nilpotent thickening: it is an isomorphism in non-positive degrees, and its kernel on homotopy groups consists of nilpotent torsion elements.  But these extra elements in $\bS$ yield some striking consequences:
\begin{enumerate}
\item Commutative algebras over $\bS_p$ admit a natural Frobenius \cite[\S IV.1]{NS}.
\item There is a natural \emph{chromatic} filtration on $\bS_p$:
\[
\bS_p = \varprojlim_n L_n \bS_p
\]
which interpolates between mixed characteristic phenomena over $\bS_p$ and characteristic zero phenomena over $L_0 \bS_p = \Q_p$.  The above Frobenius ``moves filtration by one'' in an appropriate sense.
\end{enumerate}
These features have seen recent application in the study of mixed characteristic phenomena.  For instance, the Frobenius of (1) underlies Bhatt--Morrow--Scholze's work on $p$-adic Hodge theory \cite{BMS2}, and (2) has led to the construction of some new quantum groups by Yang--Zhao \cite{YangZhao}, realizing earlier character formulas of Lusztig \cite{Lusztig87, Lusztig15}.  We hope that studying analogues of Brou\'{e}'s conjecture over $\bS_p$ may shed light on the original form of the conjecture.

\subsection{$G$-spectra} 
\label{intro:G-spectra}
Morally, a $G$-spectrum is a spectrum with an action of the group $G$. $G$-spectra ought to represent $G$-equivariant cohomology theories. Algebraic topologists have a few inequivalent ways of modeling them. In this paper we deal with ``Borel equivariant'' $G$-spectra, which in $\infty$-categorical language have a simple definition: a Borel $G$-spectrum is a functor to $\Mod(\bS)$ from the classifying space $BG$ of $G$. The $\infty$-category of Borel $G$-spectra is $\Fun(BG,\Mod(\bS))$.

The $\infty$-category of Borel $G$-spectra is an $\infty$-category of module spectra:
\[
\Fun(BG,\Mod(\bS)) \cong \LMod(\bS[G])
\]
Here, $\bS[G]$ is the ``group algebra of $G$ over $\bS$'' whose underlying spectrum is $\bigoplus_{g \in G} \bS$. We similarly have the variants
\begin{align}
\label{eq:spg}
\Fun(BG,\Mod(\bS_p)) &\cong \LMod(\bS_p[G])\\
\Fun(BG,\Mod(\bS_q)) &\cong \LMod(\bS_q[G])
\end{align}
where we first $p$-complete and then adjoint appropriate roots.  The group algebra $\bS_q[G]$, like $\bZ_q[G]$, splits into block algebras:
\[
\bS_q[G] = \bS_q[G]b_1 \times \bS_q[G]b_2 \times \cdots
\]
The blocks of $\bS_q[G]$ naturally correspond to those of $\bZ_q[G]$, since $\pi_0(\bS_q[G]b) \cong \bZ_q[G]b$. The category $\LMod(\bS_q[G])$ --- and $\LModft{\bS_q[G]}$, see below --- splits up in the same way:
\begin{equation}\label{eqn:blocksplit}
{\arraycolsep=1pt
\begin{array}{lclllll}
\LMod(\bS_q[G]) &\cong & \LMod(\bS_q[G]b_1) & \times & \LMod(\bS_q[G]b_2) & \times & \cdots \\
\LModft{\bS_q[G]} &\cong & \LModft{\bS_q[G]b_1} & \times & \LModft{\bS_q[G]b_2} & \times & \cdots
\end{array}}
\end{equation}

Let's discuss $\LModft{\bS_q[G]}$, which we propose as a spectral analog of $D^b(\bZ_q[G])^{\mrm{fg}}$. Let $\Mod(\bS_q)^{\omega}$ denote the full subcategory of $\Mod(\bS_q)$ spanned by compact objects \S\ref{subsec:finite-k}. Then $\LModft{\bS_q[G]}$ is the full subcategory of $\LMod(\bS_q[G])$ spanned by modules whose underlying $\bS_q$-module is compact. Similar to \eqref{eq:spg} we have 
\[
\Fun(BG,\Mod(\bS_q)^{\omega}) \cong \LModft{\bS_q[G]}
\]
In a way $\LModft{\bS_q[G]}$ is a difficult category to work with, because it is not known whether it has a finite set of generators. (An alternative, $\LMod(\bS_q[G])^{\omega}$, is by definition generated by $\bS_q[G]$). Nevertheless we are able to prove some cases (the cyclic defect case) of the obvious $\LMod^{\mathrm{ft}}$-analog of Brou\'e's conjecture:

\subsection{Brou\'{e}'s Conjecture for $G$-spectra}
\label{intro:maintheorem}

Since $\bS_q[G]$ and $\bZ_q[G]$ have the same blocks, each block of defect $D$ of $\bS_q[G]$ has a Brauer corresponding block of $\bS_q[N_G(D)]$: if $\bZ_q[N_G(D)]b'$ is the Brauer correspondent of $\bZ_q[G]b$, then $\bS_q[N_G(D)]b'$ is the Brauer correspondent of $\bS_q[G]b$. The natural analog of Brou\'e's conjecture for $G$-spectra would be a positive answer to the following question:

\begin{qst*}
Suppose $b$ is a block of $\bS_q[G]$ with abelian defect group $D \subset G$, and that $b'$ is the corresponding block of $\bS_q[N_G(D)]$. Then there is an equivalence of stable $\infty$-categories
\[
\LMod(\bS_q[G]b) \cong \LMod(\bS_q[N_G(D)]b')
\]
which restricts to an equivalence between the full subcategories $\LModft{\bS_q[G]b}$ and $\LModft{\bS_q[N_G(D)]b'}$. (This last statement about finite type subcategories is in fact automatic, by Proposition \ref{prop: finiteness}.)
\end{qst*}

We are not quite bold enough to state this as a conjecture, but we prove the following as Theorem \ref{thm:last}, answering the Question in the case where the defect group is \emph{cyclic}. 

\begin{thm*}
Suppose $b$ is a block of $\bS_q[G]$ with defect $D \subset G$, and that $b'$ is the corresponding block of $\bS_q[N_G(D)]$. If $D$ is cyclic, then there is an equivalence of stable $\infty$-categories
\[
\LMod(\bS_q[G]b) \cong \LMod(\bS_q[N_G(D)]b')
\]
which restricts to an equivalence between the full subcategories $\LModft{\bS_q[G]b}$ and $\LModft{\bS_q[N_G(D)]b'}$.
\end{thm*}

\subsection{Rickard vs. Rouquier}
\label{subsec:rvr}
There are two old proofs of Brou\'e's conjecture for blocks of cyclic defect, one by Rickard \cite{rickard} and one by Rouquier \cite{rouquier-cyclic}. Our proof is an adaptation of Rouquier's argument.

{\it Rickard's proof.} The Brauer tree of a block of cyclic defect has a little bit of extra structure: a ribbon structure and a distinguished vertex labeled by an integer called its ``multiplicity.'' In the case of $\PSL_2(\bF_7)$ \eqref{eq:treeG} or its Borel subgroup \eqref{eq:treeB}, the ribbon structure is the embedding in the page, the distinguished vertex is $\mathbf{3}\overline{\mathbf{3}}$, and the multiplicity is $2$.

Very few trees come from blocks. But given any tree with these decorations, Rickard defined by generators and relations an associative algebra that is Morita equivalent to a block when the tree is the Brauer tree of that block. Then, he proved that any two of these tree algebras have the same derived category, as long as the trees have the same number of edges and the integer called ``multiplicity'' is the same. An elementary argument shows that these numbers match for Brauer corresponding blocks of $G$ and $N_G(D)$.

In higher algebra, it takes more work than ``generators and relations'' to define an associative ring spectrum. For this reason we have not generalized Rickard's proof, but it might be possible and interesting to do so.

{\it Rouquier's proof.} Morita theory tells us that functors between module categories can be given by bimodules. When $D$ is cyclic, Rouquier found a very explicit $(G,N_G(D))$-bimodule that gives Brou\'e's equivalence.

This bimodule and its inverse bimodule have a simple structure: they are two-term complexes of summands of permutation $G \times N_G(D)^{\op}$-modules (the inverse is a two-term complex of summands of permutation $N_G(D) \times G^{\op}$-modules), one term of which is projective. Because of this simple structure, it is easy to find $\bS_p$-versions of these bimodules.

More generally, Brou\'e's conjecture is known for many noncyclic defect groups, usually by constructing bimodules. But at present we do not know how to lift these known bimodules to $\bS_p$; it seems to be a difficult problem.

\section{Associative algebras, modules, and bimodules in spectra} \label{sec:two}

In this section, we review some basic Morita theory over the spectrum analog of a finite-dimensional algebra.  The relevant material is developed in detail in \cite[Ch. 4]{HA}, which we draw heavily from and we refer the reader to for further details.  
In \S\ref{subsec:not}, we review some basic notations from the theory of $\infty$-categories.  
In \S\ref{subsec:finite-k}--\S\ref{subsec:right-adjoint}, we discuss the theory of bimodules and its interactions with two finiteness conditions: ``compact'' modules and ``finite type'' modules.
In \S\ref{subsec:change-of-rings} we prove the main theorem of this section, which will be used in \S\ref{sec:rouquier} to lift certain $\bZ$-linear equivalences of $\infty$-categories to $\bS$-linear equivalences.  

In \S\ref{subsec:Knlocalalg}--\S\ref{subsec:Knlocalfun}, we discuss $K(n)$-local algebras. When $A$ and $B$ are group algebra over a $K(n)$-local algebra $k$, the theory of \emph{ambidexterity} provides a larger class a functors $\LModft{A} \to \LModft{B}$. This material is not used in our proof of \S\ref{intro:maintheorem}.

\subsection{$\infty$-categorical notation}
\label{subsec:not}
We write $\cS$ for the $\infty$-category of spaces, $\bS$ for the sphere spectrum, and $\Mod(\bS)$ for the $\infty$-category of spectra. If $x$ and $y$ are objects of the $\infty$-category $\cC$, we write $\Maps(x,y) \in \cS$ for the space of maps between them, and $[x,y]$ for $\pi_0 \Maps(x,y)$. If $\cC$ is stable, then we write $\underline{\Maps}(x,y)$ for the corresponding mapping spectrum, whose associated infinite loop space is $\Maps(x,y)$. We write $\Sigma$ for the suspension functor in a stable $\infty$-category.

We write $\Fun(\cC,\cD)$ for the $\infty$-category of functors between $\infty$-categories $\cC$ and $\cD$. The full subcategory spanned by functors that have right adjoints is $\Fun^L(\cC,\cD)$. If $\cC$ and $\cD$ are stable, the full subcategory spanned by functors that preserve finite limits and colimits is $\Fun^{\ex}(\cC,\cD)$.

We write $BG$ for the classifying space of a finite group $G$, which we regard as equipped with a natural basepoint $\pt \to BG$. If $G$ is a finite group and $c$ is an object of the $\infty$-category $\cC$, then an action of $G$ on $c$ is by definition a functor $BG \to \cC$ whose value at the basepoint is $c$. We write
\[
c^{hG} := \varprojlim_{BG} c \qquad c_{hG} := \varinjlim_{BG} c
\]
for the homotopy fixed points and homotopy quotient of a $G$-object $c$, when these limits and colimits exist in $\cC$.

If $\cC =\Mod(\bS)$, there is a more sophisticated notion of a $G$-object in $\cC$: the notion of a ``genuine'' $G$-spectrum. This paper mostly does not touch the genuine theory except in \S\ref{sec:strong-segal}.

\subsection{Finiteness conditions for modules over a commutative ring spectrum}
\label{subsec:finite-k}
In this paper, a \emph{commutative ring spectrum} is an $E_{\infty}$-algebra object in $\Mod(\bS)$. Let $k$ be a commutative ring spectrum. Write $(\Mod(k), \otimes_k)$ for the symmetric monoidal stable $\infty$-category of $k$-module spectra.  The mapping spectrum $\underline{\Maps}(M,N)$ attached to two objects of $\Mod(k)$ has a natural $k$-module structure.

The following conditions are equivalent in $\Mod(k)$ \cite[Proposition 7.2.4.4]{HA}:
\begin{enumerate}
\item $M$ is \emph{compact}, i.e., the functor $\Maps(M,-):\Mod(k) \to \cS$ commutes with filtered colimits.
\item $\underline{\Maps}(M,-)$, regarded either as a functor $\Mod(k) \to \Mod(\bS)$ or $\Mod(k) \to \Mod(k)$, commutes with direct sums. \cite[Prop. 15.1]{DAGI}
\item $M$ is \emph{perfect} i.e., it is a summand of an object which carries a finite filtration whose subquotients have the form $\Sigma^n k$.
\item $M$ is \emph{dualizable}, i.e., there is a second object $M^*$ and a pair of maps $k \to M^* \otimes_k M$ and $M \otimes_k M^* \to k$ such that the composite
\[
M = M \otimes_k k \to M \otimes_k (M^* \otimes_k M) = (M \otimes_k M^*) \otimes_k M \to k \otimes_k M = M
\]
is an isomorphism.
\end{enumerate}

We will write $\Modomega{k} \subset \Mod(k)$ for the full subcategory spanned by the compact objects. For example, when $k = \bZ$, the homotopy category of $\Mod(\bZ)$ is the ``traditional'' unbounded derived category of abelian groups, and the homotopy category of $\Mod(\bZ)^{\omega}$ is the full subcategory spanned by bounded complexes with finitely generated homology groups. 

Let $\Cat_{\infty}^{\ex}$ denote the $\infty$-category of small stable $\infty$-categories and exact functors.  Then $\Cat_{\infty}^{\ex}$ admits a canonical symmetric monoidal structure $\otimes$, which comes equipped with a universal map $\cC \times \cD \to \cC \otimes \cD$ which commutes with finite colimits separately in each variable.  We also have reason to consider the ``large'' setting of presentable $\infty$-categories.  The $\infty$-category $\PrL$ of presentable $\infty$-categories also admits a symmetric monoidal structure $\otimes$, which commutes with (arbitrary) colimits separately in each variable \cite[Corollary 4.8.1.4]{HA}.  

\begin{exm}
We have $\Modomega{k} \in \Cat_{\infty}^{\ex}$ (as well as $\LModomega{A}$, $\LModft{A}$ to be introduced in the following section) and 
\[
\Mod(k) \cong \Ind(\Modomega{k}) \in \PrL.
\]
More generally, the Ind-category construction determines a symmetric monoidal functor $\Ind : \Cat_{\infty}^{\ex} \to \PrL$.  
\end{exm}

We will call a stable $\infty$-category $\cC \in \Cat_{\infty}^{\ex}$ (resp. $\cC \in \PrL$) \emph{$k$-linear} when it is endowed with an action of the algebra $\Modomega{k} \in \Cat_{\infty}^{\ex}$ (resp. $\Mod(k) \in \PrL$) (cf. \cite[Def. 6.2]{DAGVII}). We denote the bifunctor $\Modomega{k} \otimes \cC \to \cC$ (resp. $\Mod(k) \otimes \cC \to \cC$) by $\otimes_k$. The mapping spectrum $\underline{\Maps}(x,y)$ between objects of a $k$-linear category has a $k$-module structure, which is not always perfect but which represents the functor
\[
\Maps(- \otimes_k x,y) : \Mod(k)^{\omega, \op} \to \cS.
\]
If $\cC$ and $\cD$ are $k$-linear categories, we write $\Fun_k^{\ex}(\cC,\cD)$ for the $\infty$-category of exact functors that preserve this action.

\subsection{Finiteness conditions on associative algebra spectra and their modules}
\label{subsec:fin-left-mod}
For a commutative ring spectrum $k$, a \emph{$k$-algebra spectrum} will mean an associative $k$-algebra spectrum --- that is, an $E_1$-algebra over $k$.  We write $\LMod(A)$ for the $\infty$-category of left $A$-module spectra. It is presentable, stable, and has a $k$-linear structure. 

We write $\RMod(A)$ and $\Bimod(A,B)$ for the $\infty$-categories of right $A$-modules and of $(A,B)$-bimodules respectively.  The theory of right modules and bimodules is related to the theory of left modules via the natural equivalences $\RMod(A) \cong \LMod(A^{\op})$ and $\Bimod(A,B) \cong \LMod(A \otimes_k B^{\op})$ \cite[Proposition 4.3.2.7, Proposition 4.6.3.11]{HA}.

For an object $M \in \LMod(A)$, the following are equivalent \cite[Proposition 7.2.4.4]{HA}:
\begin{enumerate}
\item $\Maps(M,-):\LMod(A) \to \cS$ commutes with filtered colimits.
\item $\underline{\Maps}(M,-)$, regarded either as a functor $\LMod(A) \to \Mod(k)$ or $\LMod(A) \to \Mod(\bS)$, commutes with direct sums \cite[Prop. 15.1]{DAGI}.

\item $M$ is a summand of an object that has a finite filtration whose associated graded pieces have the form $\Sigma^n A$.
\end{enumerate}
A module obeying these conditions is called \emph{perfect}. We denote the full subcategory spanned by perfect left $A$-modules by $\LModomega{A}$. 

If $A$ is perfect as a $k$-module, then it is natural to consider a weaker finiteness condition: we say that $M \in \LMod(A)$ has \emph{finite type} if its underlying $k$-module belongs to $\Modomega{k}$. Write $\LModft{A}$ for the full subcategory spanned by finite type $A$-modules. We have a containment

\[
\LModomega{A} \subset \LModft{A} \text{ if and only if }A \in \Modomega{k}.
\]

\begin{rmk}
If $k = \bZ$ and $A$ is an associative ring, the homotopy category of $\LModomega{A}$ coincides with the category of bounded complexes of projective modules and chain homotopy classes of maps between them. For such an $A$, the condition that the underlying additive group of $A$ is finitely generated is equivalent to the condition that $A \in \Modomega{k}$, in which case the homotopy category of $\LModft{A}$ is the same as $D^b(\text{f.g. left $A$-modules})$.
\end{rmk}

\subsection{Projective $A$-module spectra}
\label{subsec:proj}

We call an object of $\LMod(A)$ a \emph{free module} if it is isomorphic to a direct sum of objects of the form $\Sigma^d A$. Write $\mathrm{Free}(A) \subset \LMod(A)$ and $\mathrm{Free}(A)^{\omega} \subset \LModomega{A}$ for the full subcategories spanned by free modules.

We call an object of $\LMod(A)$ a \emph{projective module} if it is a direct summand of a free module. Write $\mathrm{Proj}(A) \subset \LMod(A)$ and $\mathrm{Proj}(A)^{\omega} \subset \LModomega{A}$ for the full subcategories spanned by projective modules. 

The free modules $\Sigma^d A$ represent the functor 
$
\pi_d:h\LMod(A) \to \mathrm{Ab}
$ where the domain is homotopy category of $\LMod(A)$ and the codomain is the usual $1$-category of abelian groups. Furthermore, for any $M \in \LMod(A)$, the graded abelian group $\pi_*(M) := \bigoplus_{i \in \bZ} \pi_i(M)$ has the structure of a graded $\pi_*(A)$-module, and the natural map
\begin{equation}
\label{eq:sigmadrep}
[\Sigma^d A,M] \to \Hom_{\pi_*(A)}(\Sigma^d \pi_*(A),\pi_*(M))
\end{equation}
is an isomorphism. The codomain in \eqref{eq:sigmadrep} denotes the set of homomorphisms of graded abelian groups that are compatible with the grading and the $\pi_*(A)$-module structure, and $\Sigma^d \pi_*(A)$ denotes a grading shift. We record two consequences of this observation:

\begin{prop}
\label{prop:barP}
Let $\bar{P}$ be a graded abelian group equipped with a left $\pi_*(A)$-module structure. If $\bar{P}$ is projective, then there is an object of $\mathrm{Proj}(A)$ such that $\pi_*(P) \cong \bar{P}$ as graded $\pi_*(A)$-modules.
\end{prop}

\begin{proof}
First suppose that $\bar{P} = \bar{F}$ is free: $\bar{F} \cong \bigoplus_{i \in I} \Sigma^{d_i} \pi_{*}(A)$.  Then we may take $P= F = \bigoplus_{i \in I} \Sigma^{d_i} A$. 

In general, $\bar{P}$ is the image of an idempotent endomorphism $e:\bar{F} \to \bar{F}$. The isomorphism \eqref{eq:sigmadrep} shows this lifts to an idempotent endomorphism of $F$, and idempotents in $\LMod(A)$ split by \cite[Lemma 1.2.4.6]{HA}.
\end{proof}

\begin{prop}
\label{prop:PM}
Let $P \in \Proj(A)$ and $M \in \LMod(A)$. Then the map
\begin{equation}
\label{eq:PM}
[P,M] \to \Hom_{\pi_*(A)}(\pi_*(P),\pi_*(M))
\end{equation}
is an isomorphism.
\end{prop}

Note one consequence of Proposition \ref{prop:PM} is that the lift in Proposition \ref{prop:barP} of $\bar{P}$ to $P$ is unique up to isomorphism.

\begin{proof}
This is a weaker version of \cite[Corollary 7.2.2.19]{HA}; it can be proved as follows.
Since $P$ is a summand of a free module $F = \bigoplus_{i} \Sigma^{d_i} A$, the domain of \eqref{eq:PM} is a retract of $\prod_i [\Sigma^{d_i}A,M]$ and the codomain is a retract of $\prod_i \Hom_{\pi_*(A)}(\Sigma^{d_i} \pi_*(A),\pi_*(M))$. The retractions that are induced by an idempotent endomorphism of $F$ commute with the maps \eqref{eq:PM} for $P$ and for $F$. Since the retract of an isomorphism is an isomorphism, we are reduced to proving it for free modules, and then further reduced to the case $P = \Sigma^d A$, which is \eqref{eq:sigmadrep}.
\end{proof}

\subsection{Finiteness conditions for functors}
\label{subsec:fin-con-fun}
Let $A$ and $B$ be $k$-algebra spectra. Then $\LMod(A)$ and $\LMod(B)$ are $k$-linear presentable stable $\infty$-categories, and we can consider $k$-linear exact functors between them. For such a functor, the following are equivalent:
\begin{enumerate}
\item $F:\LMod(A) \to \LMod(B)$ preserves filtered colimits.
\item $F:\LMod(A) \to \LMod(B)$ preserves all colimits.
\item $F:\LMod(A) \to \LMod(B)$ preserves (arbitrary) direct sums.
\item $F$ has a right adjoint.
\end{enumerate}
We write $\Fun_k^L(\LMod(A),\LMod(B)) \subset \Fun_k^{\ex}(\LMod(A),\LMod(B))$ for the full subcategory spanned by colimit-preserving functors.

If $F$ is any such functor, then $F(A)$ supports a $(B,A)$-bimodule structure, $F$ is isomorphic to $F(A) \otimes_A -$, and

\begin{equation}
\label{eq:role-of-bimodules}
\Fun^L_k(\LMod(A),\LMod(B)) \cong \LMod(B \otimes_k A^{\op})
\end{equation}
 \cite[Proposition 4.8.4.1]{HA}.  Furthermore, the restriction to $\LModomega{A}$ induces an equivalence:
\begin{equation}
\label{eq:using-ind}
\Fun^L_k(\LMod(A),\LMod(B)) \xrightarrow{\sim} \Fun^{\ex}_k(\LModomega{A},\LMod(B)).
\end{equation}

\begin{prop}
\label{prop:FA-perf-B}
Let $A$ and $B$ be $k$-algebra spectra and let $F$ be a colimit-preserving $k$-linear functor $\LMod(A) \to \LMod(B)$.
The following are equivalent:
\begin{enumerate}
\item $F(A)$ is perfect as a left $B$-module
\item $F$ carries $\LModomega{A}$ into $\LModomega{B}$
\end{enumerate}
\end{prop}
In other words after \eqref{eq:role-of-bimodules} we have
\begin{equation}
\Fun^\ex_k(\LModomega{A},\LModomega{B}) \cong \left(\begin{array}{c}\text{ full subcategory of  } \\
\text{
 $\Bimod(B,A)$ spanned } \\
 \text{by bimodules that are } \\
 \text{perfect as left $B$-modules}
\end{array}
\right)
\end{equation}

\begin{proof}
Condition (2) implies condition (1) because $A$ is an object of $\LModomega{A}$. If $M \in \LModomega{A}$ has a filtration
\[
0 \to M_{\leq 0} \to M_{\leq 1} \to \cdots \to M_{\leq n} = M
\]
with $M_{\leq i}/M_{\leq i-1} \cong \Sigma^{d_i} A$, then $F(M)$ has a filtration whose subquotients are isomorphic to $\Sigma^{d_i} F(A)$. If $M'$ is a summand of such an $M$ then $F(M')$ is a summand of $F(M)$. Since $F(A)$ is perfect so is $\Sigma^{d_i} F(A)$, therefore so is $F(M)$, and therefore so is $F(M')$ --- this shows that (1) implies (2).
\end{proof}

\begin{prop}
\label{prop:FA-right-perfect}
Let $A$ and $B$ be $k$-algebra spectra and let $F$ be a  $k$-linear colimit-preserving functor $\LMod(A) \to \LMod(B)$.
Suppose that $F(A)$ is perfect as a right $A$-module and as a $k$-module. Then $F$ carries $\LModft{A}$ into $\LModft{B}$.
\end{prop}

\begin{rmk} 
When we consider $\LModft{A}$, it will typically be the case that $A$ is perfect as a $k$-module --- in that case, when $F(A)$ is perfect as a right $A$-module it is automatically perfect as a $k$-module. 
\end{rmk}

\begin{proof}
Since $F(A) \otimes_A M$ belongs to $\LModft{B}$ exactly when the underlying $k$-module belongs to $\Modomega{k}$, it suffices to show that the bifunctor $(-) \otimes_A (-):\RMod(A) \times \LMod(A) \to \Mod(k)$ carries $\RMod(A)^{\omega} \times \LModft{A}$ into $\Modomega{k}$.

Let $N \in \RMod(A)^{\omega}$ and let $M \in \LModft{A}$. Fix a filtration
\[
0 \to N_{\leq 0} \to \cdots \to N_{\leq n} = N
\]
such that $N_{\leq i}/N_{\leq i-1}$ is a suspension of $A$ (as a right $A$-module). Since $\Sigma^{d_i} A \otimes_A M = \Sigma^{d_i} M$, and $M$ is perfect as a $k$-module, it follows that $N \otimes_A M$ is perfect as a $k$-module. If $N'$ is a summand of $N$ then $N' \otimes_A M$ is a summand of $N \otimes_A M$ so it is also perfect as a $k$-module.

\end{proof}

\begin{rmk}
It is tempting to guess that there are weaker hypotheses than those of Prop. \ref{prop:FA-right-perfect} that would guarantee that $F$ carries $\LModft{A}$ to $\LModft{B}$. The following example shows some limitations --- at least, that it can fail even when $F(A)$ is perfect or finite type as a left $B$-module. Let $A = k[G]$, $B = k$, and let $M = k$ be the trivial $(B,A)$-bimodule. Tensoring with $M$ is isomorphic to the colimit-preserving functor
\[
(-)_{hG}:\LMod(k[G]) \to \LMod(k),
\]
i.e. to homotopy $G$-coinvariants. The functor carries the trivial module to the $k$-homology of $BG$, which often does not lie in $\LMod(k)^{\omega} = \LMod(k)^{\mathrm{ft}}$. For example, the $k$-homology of $BG$ is not perfect when $G$ is finite and nontrivial and $k = \bS$ or $\bZ$, and if $p$ divides the order of $G$ then it is not perfect when $k = \bS_p$, $\bZ_p$, or $\bF_p$. Thus, it does not carry $\LModft{k[G]}$ into $\LModft{k}$, even though $M$ is perfect and finite type as a left $B$-module. (Meanwhile, note that since $M$ is not perfect as a right $A$-module, it also does not carry $\LMod(k[G])^{\omega}$ into $\LMod(k)^{\omega}$.)

For some values of $k$, the Greenlees-Sadofsky ``Tate vanishing'' or Hopkins-Lurie ``ambidexterity'' can repair this issue in a significant way \S\ref{subsec:Knlocalfun}.
\end{rmk}

\subsection{Finiteness conditions for the right adjoint functor}
\label{subsec:right-adjoint}
Let $F:\LMod(A) \to \LMod(B)$ be a colimit-preserving $k$-linear functor, so that it has a $k$-linear right adjoint $G$. Just as $F(M) \cong F(A) \otimes_A M$ for $F(A)$ endowed with its natural $(B,A)$-bimodule structure \eqref{eq:role-of-bimodules}, $G(M)$ is given by the formula
\begin{equation}
\label{eq:GMFAM}
G(M) = \underline{\Maps}_{\LMod(B)}(F(A),M)
\end{equation}
The right $A$-module structure of $F(A)$ induces a left $A$-module structure on \eqref{eq:GMFAM}. Note in particular that $G(B) \cong \underline{\Maps}_{\LMod(B)}(F(A),B)$. 

\begin{prop}
\label{prop:right-adjoint}
Let $A$ and $B$ be $k$-algebra spectra, and let $F:\LMod(A) \to \LMod(B)$ be a colimit-preserving $k$-linear functor. Let $G:\LMod(B) \to \LMod(A)$ be its right adjoint. The following are equivalent:
\begin{enumerate}
\item $F(A)$ is perfect as a left $B$-module
\item $F$ carries $\LModomega{A}$ into $\LModomega{B}$
\item $G$ preserves colimits
\item There is a natural isomorphism $G(M) \cong G(B) \otimes_B M$
\end{enumerate}
\end{prop}

\begin{proof}
(1) and (2) are equivalent by Prop. \ref{prop:FA-perf-B}. (1) and (3) are equivalent by \eqref{eq:GMFAM}, the fact that colimits in $\LMod(B)$ are computed in $\Mod(k)$, and \S\ref{subsec:fin-left-mod}. (3) and (4) are equivalent by \S\ref{subsec:fin-con-fun}.
\end{proof}

As a corollary we have the following criterion for a colimit-preserving $k$-linear functor to induce a pair of adjoint functors between $\LModft{A}$ and $\LModft{B}$:
 
\begin{prop}
\label{prop:AftBft}
Suppose that $A$ and $B$ are perfect over $k$,  let $F:\LMod(A) \to \LMod(B)$ be a colimit-preserving $k$-linear functor, and let $G$ be its right adjoint. Suppose that $F(A)$ is perfect as a left $B$-module and as a right $A$-module. Then $F$ carries $\LModft{A}$ into $\LModft{B}$ and $G$ carries $\LModft{B}$ into $\LModft{A}$.
\end{prop}

\begin{proof}
Since $F(A)$ is perfect as a right $A$-module (and since $A$ is perfect over $k$), $F$ carries $\LModft{A}$ into $\LModft{B}$ by Prop. \ref{prop:FA-right-perfect}. Since $F(A)$ is perfect as a left $B$-module, the adjoint $G$ is colimit-preserving by \Cref{prop:right-adjoint}, and we may detect whether $G$ carries $\LModft{B}$ into $\LModft{A}$ by applying Prop. \ref{prop:FA-right-perfect} to
\[
G(B) = \underline{\Maps}_{\LMod(B)}(F(A),B)
\]
But this is $F(A)^\vee$, which is perfect as a right $B$-module by Prop. \ref{prop:vee-dual}.
\end{proof}

\begin{rmk}
\label{rem:koszul}
Not every pair of adjoint functors $\LModft{A} \leftrightarrows \LModft{B}$ extends to a colimit-preserving functor $\LMod(A) \to \LMod(B)$. For example when $G$ is a commutative $p$-group, \cite[\S 3.6]{kug}  constructs self-equivalences of $\LModft{\KU_p[G]}$ that exchange the trivial module $\KU_p$ and the free module $\KU_p[G]$ 
\[
\KU_p \leftrightarrow \KU_p[G]
\]
Since the trivial module is not perfect, the equivalence does not preserve the subcategory $\LMod(\KU_p[G])^{\omega}$ and does not extend to a self-equivalence of $\LMod(\KU_p[G])$.
\end{rmk}

Another consequence of Prop. \ref{prop:right-adjoint} is that any $k$-linear functor $\LModomega{A} \to \LModomega{B}$ extends to a colimit-preserving functor $\LMod(A) \to \LMod(B)$ whose right adjoint is also colimit-preserving. We will want a criterion for this right adjoint to carry $\LModomega{B}$ back into $\LModomega{A}$. According to Prop. \ref{prop:FA-perf-B}, a necessary and sufficient condition is that
\begin{equation}
\label{eq:nn-sym-formula-for-G(B)}
G(B) = \underline{\Maps}_{\LMod(B)}(F(A),B) \text{ is perfect as a left $A$-module}
\end{equation}
It is perhaps hard to tell at a glance whether this is the case---for instance to give a criterion in terms of the right $A$-module structure on $F(A)$. We can give a useful criterion like that when $A$ and $B$ are both ``symmetric''; this will be discussed in \S \ref{sec: symmetric}.

\subsection{Change of rings}
\label{subsec:change-of-rings}
Let $k$ continue to denote a commutative ring spectrum (i.e. what it has been denoting since \S\ref{subsec:finite-k}). Suppose we have a second commutative ring spectrum $k'$, and a map $u:k \to k'$. Then $u$ induces a symmetric monoidal, colimit-preserving functor
\begin{equation}
\label{eq:otimeskprime}
u^*:\Mod(k) \to \Mod(k') \qquad u^*(M):= k' \otimes_k M
\end{equation}
that carries $\Modomega{k}$ into $\Mod(k')^{\omega}$.

We also denote by $u^*$ the functor induced by \eqref{eq:otimeskprime} from algebras in $\Mod(k)$ to algebras in $\Mod(k')$. If $A$ is a $k$-algebra spectrum and $A' := u^*(A)$ is the induced $k'$-algebra spectrum, $\LMod(A)$ and $\LMod(A')$ are related via the formula in $\PrL$
\begin{equation}
\label{eq:haukensis}
\LMod(A') \cong \Mod(k') \otimes_{\Mod(k)} \LMod(A).
\end{equation}

It follows that a $k$-linear equivalence between $\LMod(A)$ and $\LMod(B)$ induces a $k'$-linear equivalence between $\LMod(A')$ and $\LMod(B')$, where $B$ is a second $k$-algebra spectrum and $B' := u^*(B)$.

\begin{rmk}
We can also deduce an equivalence $\LMod(A')^{\omega} \cong \LMod(B')^{\omega}$ whenever $\LModomega{A} \cong \LModomega{B}$. But an equivalence $\LModft{A} \cong \LModft{B}$ may not induce an equivalence $\LModft{A'} \ncong \LModft{B'}$. The self-equivalences in Remark \ref{rem:koszul} provide an example with $k = \KU_p$ and $k' = \bQ_p[\beta,\beta^{-1}]$, and the map $k \to k'$ being the Chern character. If $G$ is a commutative $p$-group then $\LModft{k'[G]} = \LMod^{\omega}(k'[G])$ is semisimple and no self-equivalence of it can exchange the trivial representation for the regular representation.
\end{rmk}

We would like to study the converse problem: given a $k'$-linear equivalence $\LMod(A') \cong \LMod(B')$, can we conclude that $\LMod(A) \cong \LMod(B)$?

\begin{prop}\label{prop:conservative}
Let $A$ and $B$ be $k$-algebras which are perfect over $k$. Let $F(A)$ be a $(B,A)$-bimodule which is perfect separately as a left $B$-module and as a right $A$-module. Let $u:k \to k'$ be a commutative $k$-algebra and put $A' := u^*(A)$ and $B' = u^*(B)$. 
If
\begin{equation}
\label{eq:nakayama}
u^*:\Modomega{k} \to \Mod(k')^{\omega} \text{ is conservative}
\end{equation}
and the induced functor
\begin{equation}
\label{eq:nosema}
F':\LMod(A') \to \LMod(B')
\end{equation}
is an equivalence of $k'$-linear $\infty$-categories, then $\LMod(A) \to \LMod(B)$ is an equivalence of $k$-linear $\infty$-categories.
\end{prop}

\begin{rmk}
The hypothesis \eqref{eq:nakayama} applies when $k$ is a discrete local ring and $u:k \to k'$ is the quotient by the maximal ideal. More generally if $k$ and $k'$ are discrete rings, then \eqref{eq:nakayama} holds if and only if the image of $\Spec(k') \to \Spec(k)$ contains all the closed points.

Hypothesis \eqref{eq:nakayama} also applies to the truncation map $\bS \to \bZ$: this is one formulation of the  Whitehead theorem for homology groups. 

\end{rmk}

\begin{proof}
After Prop. \ref{prop:right-adjoint} and \eqref{eq:nn-sym-formula-for-G(B)} the right adjoint functor to $F(A) \otimes_A -$ is
\[
\underline{\Maps}(F(A),B) \otimes_B - .
\]
The counit for the adjunction is induced by a map
\begin{equation}
\label{eq:counit-in-proof}
F(A) \otimes_A \underline{\Maps}(F(A),B) \to B
\end{equation}
of $(B,B)$-bimodules, and the unit is induced by a map of $(A,A)$-modules
\begin{equation}
\label{eq:unit-in-proof}
A \to \underline{\Maps}{(F(A),B)} \otimes_B F(A)
\end{equation}
The induced $k'$-linear functors between $\LMod(A')$ and $\LMod(B')$ ($F'$ \eqref{eq:nosema} and its adjoint $G'$) are isomorphic to 
\[
u^* F(A) \otimes_{A'} - \qquad \text{and} \qquad u^*(\underline{\Maps}(F(A),B)) \otimes_{B'} -
\]

By assumption, $F'$ is an equivalence. Its inverse equivalence must be $G'$ and the maps
\[
u^*F(A) \otimes_{A'} u^*(\underline{\Maps}(F(A),B)) \to B' \qquad A' \to u^*\underline{\Maps}{(F(A),B)} \otimes_{B'} u^*F(A)
\]
are isomorphisms of bimodules, and since $u^*$ is a symmetric monoidal functor so are
\[
u^*(F(A) \otimes_A \underline{\Maps}(F(A),B)) \to B' \qquad A' \to u^*(\underline{\Maps}{(F(A),B)} \otimes_{B} F(A))
\] 
Then \eqref{eq:counit-in-proof} and \eqref{eq:unit-in-proof} are isomorphisms by \eqref{eq:nakayama}.

\end{proof}

\subsection{Duality}
There are two dualities between $\Bimod(A,B)$ and $\Bimod(B,A)$, ``left'' and ``right'' \cite[\S 4.6.2, \S 4.6.4]{HA}. In this section we review these concepts in the special case $B=k$ (dualities between $\LMod(A) \cong \Bimod(A,k)$ and $\RMod(A) \cong \Bimod(k,A)$) and make some remarks.

\subsubsection{Duality of perfect $A$-modules} For any $k$-algebra spectrum $A$ and any left $A$-module $M$, the $(A,A)$-bimodule structure on $A$ endows $\underline{\Maps}_{\LMod(A)}(M,A)$ with the structure of a right $A$-module. We denote it by $M^\vee$; the construction is a contravariant functor
\[
\LMod(A)^{\op} \to \RMod(A):M \mapsto M^\vee
\]

\begin{prop}
\label{prop:vee-dual}
For any $k$-algebra spectrum $A$, the functor
$\LMod(A)^{\op} \to \RMod(A):M \mapsto M^\vee$
restricts to an equivalence
\[
\left(\LModomega{A}\right)^{\op} \cong \RMod(A)^{\omega}
\]
\end{prop}

\begin{proof}
If $M \in \LMod(A)$ has a filtration
\[
0 \to M_{\leq 0} \to M_{\leq 1} \to \cdots \to M_{\leq n} = M
\]
by left $A$-modules, such that $M_{\leq i}/M_{\leq i-1} \cong \Sigma^{d_i} A$ for each $i$, then $M^{\vee}$ has a filtration by right $A$-modules:
\[
(M/M_{\leq n})^\vee \to \cdots (M/M_{\leq 1})^\vee \to (M/M_{\leq 0})^\vee \to M^\vee
\]
whose graded pieces have the form $\Sigma^{-d_i} A^\vee$. Since $A^\vee \cong A$, it follows that $M^\vee$ is perfect as a right $A$-module, and so is any summand of $M^\vee$. Thus $(-)^{\vee}$ carries $(\LModomega{A})^{\op}$ into $\RMod(A)^{\omega}$. The composite
\[
\RMod(A)^{\omega} = \LMod(A^{\op})^{\omega} \xrightarrow{(-)^{\vee}} \left(\RMod(A^{\op})^{\omega}\right)^{\op} = \left(\LModomega{A}\right)^{\op}
)   
\]
is the inverse functor.
\end{proof}

\subsubsection{Duality of $k$-modules}
\label{subsec:dokm}

If $M \in \LModft{A}$, then $M^*$ (the monoidal dual of $M$ regarded as a $k$-module, notation as in \S\ref{subsec:finite-k}) has the structure of a right $A$-module spectrum. If $N$ is a second $A$-module spectrum of finite type, then \cite[Prop. 4.6.2.1]{HA} the natural map
\begin{equation}
\label{eq:star-map}
\underline{\Maps}_{\LMod(A)}(M,N) \to \underline{\Maps}_{\RMod(A)}(N^*,M^*)
\end{equation}

is an equivalence of $k$-module spectra. In fact we have a commutative square:
\begin{equation}
\label{eq:lmod-rmod}
\xymatrix{
\LModft{A}^{\op} \ar[r]^-{\cong} \ar[d]_{\mathit{forget}} & \RModft(A) \ar[d]^{\mathit{forget}} \\
\left(\Modomega{k}\right)^{\op} \ar[r]_-{\cong}  & \Modomega{k}
}
\end{equation}
where the horizontal maps are given by $(-)^*$.

\begin{prop}
\label{prop:dual-of-perf}
Suppose $A$ is perfect over $k$ and that $A^*$ is perfect as a right $A$-module. Then \eqref{eq:lmod-rmod} is a fully faithful embedding of $(\LModomega{A})^{\op}$ into $\RMod(A)^{\omega}$. If $A^*$ is also perfect as a left $A$-module, then \eqref{eq:lmod-rmod} restricts to an equivalence
$
(\LModomega{A})^{\op} \cong \RMod(A)^{\omega}
$.
\end{prop}

\begin{rmk}
The example in \S\ref{ex:proj-with-infinite-inj-dimension} shows the hypothesis on $A^*$ cannot be removed. The hypothesis is easy to verify for symmetric algebras.
\end{rmk}

\begin{proof}

If $M \in \LModomega{A}$ has a filtration
\[
0 \to M_{\leq 0} \to M_{\leq 1} \to \cdots \to M_{\leq n} = M
\]
by left $A$-modules, such that $M_{\leq i}/M_{\leq i-1} \cong \Sigma^{d_i} A$ for each $i$, then $M^*$ has a filtration by right $A$-modules:
\[
(M/M_{\leq n})^* \to \cdots (M/M_{\leq 1})^* \to (M/M_{\leq 0})^* \to M^*
\]
whose graded pieces have the form $\Sigma^{-d_i} A^*$.  By assumption, $A^*$ is perfect as a right $A$-module and therefore $M^*$ is as well, and so is any summand of $M^*$.  

The functor is fully faithful by \eqref{eq:star-map}. If $A^*$ is perfect as a left $A$-module, the $(A^{\op})^*$ is perfect as a right $(A^{\op})$-module and the same argument, together with $M \cong (M^*)^*$ when $M$ is a perfect $k$-module, shows that $\RMod^{\omega}(A)^{\op} \to \LModomega{A}$ is the inverse equivalence.

\end{proof}

\subsubsection{Example}
\label{ex:proj-with-infinite-inj-dimension}
Suppose $A \in \Modomega{k}$, i.e. that $\LModomega{A} \subset \LModft{A}$. The equivalence 
\[
(\LModft{A})^{\op} \cong \RModft(A)
\]
does not always carry $(\LModomega{A})^{\op}$ into $\RMod(A)^{\omega}$. For instance, consider the case where $k$ is a field and $A$ is the four-dimensional algebra with basis $e_1,e_2,f,\epsilon$, where $e_1$ and $e_2$ are orthogonal idempotents and
\[
e_1 f e_2 = f \qquad e_2 \epsilon e_2 = 0 \qquad \epsilon f = 0 \qquad \epsilon^2 = 0
\]
Left modules over this ring are representations of the quiver
\[
\xymatrix{
\bullet \ar[r]_-f & \bullet \ar@(ur,dr)^{\epsilon} 
}\qquad \text{subject to $\epsilon^2 = 0$ and $\epsilon f = 0$}
\]
while right modules are representation of the quiver
\[
\xymatrix{
\bullet & \ar[l]_-{f^T}  \bullet \ar@(dr,ur)_{\epsilon^T} 
} \qquad \text{subject to $(\epsilon^T)^2 = 0$ and $f^T \epsilon^T = 0$}
\]
The representation
\[
\xymatrix{
k \ar[r]^-{=} & k \ar@(ur,dr)^{0}
}
\]
is projective but its dual
\[
\xymatrix{
k &  \ar[l]_{=} k \ar@(dr,ur)_{0}
}
\]
is not, nor does it have a finite projective resolution.

\subsection{Symmetric structures}\label{sec: symmetric}
In the literature on derived equivalences between group algebras and their block algebras, the natural ``symmetric structures'' on these algebras is sometimes used, e.g. \cite{rouquierblock}. The symmetric structure has some pleasant consequences for the right adjoint of a functor $D^b(A) \to D^b(B)$ that is given by a complex of $(B,A)$-bimodules. In this section we explain the analog of these consequences for algebra spectra. 

\begin{dfn}
Let $B$ be a $k$-algebra spectrum that is perfect as a $k$-module. A \emph{symmetric structure} on $B$ is an isomorphism of $(B,B)$-bimodules $B \cong B^*$.
\end{dfn}

Let $F:\LMod(A) \to \LMod(B)$ be a colimit-preserving $k$-linear functor, i.e. $F(M) = F(A) \otimes_A M$. We have already seen that $F$ carries $\LModomega{A}$ into $\LModomega{B}$ if and only if the $(B,A)$-bimodule $F(A)$ is perfect as a left $B$-module. When $B$ is perfect over $k$, we can also conclude that $F(A)$ is perfect over $k$ and consider $F(A)^*$ with its $(A,B)$-bimodule structure.

\begin{prop}
Suppose that $B$ is perfect over $k$ and is endowed with a symmetric structure. Let $F:\LMod(A) \to \LMod(B)$ be a colimit-preserving functor that carries $\LModomega{A}$ into $\LModomega{B}$. Then the right adjoint $G:\LMod(B) \to \LMod(A)$ is given by
\[
G(M) = F(A)^* \otimes_B M
\]
\end{prop}

\begin{proof}
By Prop. \ref{prop:right-adjoint}, $G$ is colimit-preserving and is given by the formula $G(M) = G(B) \otimes_B M$, and \eqref{eq:nn-sym-formula-for-G(B)} gives a formula for $G(B)$:
\[
G(B) = \underline{\Maps}_{\LMod(B)}(F(A),B)
\]
Since both $F(A)$ and $B$ are perfect over $k$, we furthermore have an isomorphism of left $A$-modules \eqref{eq:star-map}
\begin{equation}
\label{eq:star-map2}
\underline{\Maps}_{\LMod(B)}(F(A),B) \cong \underline{\Maps}_{\RMod(B)}(B^*,F(A)^*)
\end{equation}
where on the domain the left $A$-module structure is induced by the right $A$-module structure on $F(A)$, and on the codomain by the left $A$-module structure on $F(A)^*$.
The composite is 
\[
G(B) \cong \underline{\Maps}_{\RMod(B)}(B^*,F(A)^*)
\]
The bimodule isomorphism $B^* \cong B$ is also a right $B$-module isomorphism and it induces a further isomorphism
\[
G(B) \cong \underline{\Maps}_{\RMod(B)}(B,F(A)^*) \cong F(A)^*
\]
where the second isomorphism $\underline{\Maps}_{\RMod(B)}(B,N) \cong N$ holds for any right $B$-module $N$.
\end{proof}

When both $A$ and $B$ are perfect over $k$, and both are endowed with symmetric structures, we can make some further conclusions:

\begin{cor}\label{prop:restrict-to-ft}
Suppose that both $A$ and $B$ are perfect over $k$ and can be endowed with symmetric structures. Let $F:\LMod(A) \to \LMod(B)$ be a colimit-preserving functor that carries $\LModomega{A}$ into $\LModomega{B}$. Then the following are equivalent:
\begin{enumerate}
\item $F(A)$ is perfect as a right $A$-module
\item $F(A)^*$ is perfect as a left $A$-module
\item The right adjoint $G(M) = F(A)^* \otimes_B M$ carries $\LModomega{B}$ into $\LModomega{A}$
\end{enumerate}
When these equivalent conditions hold, $F$ carries $\LModft{A}$ into $\LModft{B}$ and $G$ carries $\LModft{B}$ into $\LModft{A}$.
\end{cor}

\begin{proof}
(1) and (2) are equivalent by Prop. \ref{prop:dual-of-perf}, and the equivalence of (2) and (3) is immediate from the formula $G(M) = F(A)^*\otimes_B M$.

That $F$ carries $\LModft{A}$ into $\LModft{B}$ is immediate from Prop. \ref{prop:FA-right-perfect}. That $G$ carries $\LModft{B}$ into $\LModft{A}$ follows from the same Proposition applied to $G(B)$.
\end{proof}

\subsection{$K(n)$-local algebras}
\label{subsec:Knlocalalg}

While Proposition \ref{prop:FA-perf-B} is a complete description of functors $\LMod(A)^{\omega} \to \LMod(B)^{\omega}$ in terms of $(B,A)$-bimodules, Proposition \ref{prop:FA-right-perfect} only gives a sufficient condition for a bimodule to determine a functor $\LModft{A} \to \LModft{B}$. In this and the next section, we discuss how we can do better when $A$ and $B$ are group algebras for finite groups and $k$ obeys a technical condition (\Cref{dfn:good}) related to the theory of \emph{$K(n)$-local} spectra; the main result is \Cref{cor:Knfinal}. This material is not used in the proof of \S\ref{intro:maintheorem}.

For $n \geq 1$ let $K(n) = K(n)_p$ be the $n$th Morava $K$-theory spectrum at the prime $p$, with
\[
\pi_i(K(n)) = \begin{cases}
\bZ/p & \text{if $i \in (2p^n -2)\bZ$}\\
0 & \text{otherwise.}
\end{cases}
\]
An object $c$ of (any) presentable stable $\infty$-category $\cC$ is called \emph{$K(n)$-acyclic} if $K(n) \otimes_{\bS} c = 0$, and $d \in \cC$ is called \emph{$K(n)$-local} if $\Maps(c,d)$ is contractible whenever $c$ is $K(n)$-acyclic. The full subcategory of $K(n)$-local objects of $\cC$ is the essential image of an idempotent functor $L_{K(n)}: \cC \to \cC$ and is denoted $L_{K(n)}\cC$. 

We have two useful formulas for the limit and for the colimit of a diagram $D:I \to L_{K(n)}\cC$:
\[
\varprojlim D = \varprojlim (L_{K(n)} \cC \hookrightarrow \cC)\circ D \qquad \varinjlim D = L_{K(n)} \varinjlim  (L_{K(n)} \cC \hookrightarrow \cC)\circ D
\]
In other words, the inclusion of $L_{K(n)} \cC \to \cC$ preserves limits and the localization functor $\cC \to L_{K(n)} \cC$ preserves colimits.

We are interested in the categories $L_{K(n)} \LMod(A)$, where $A$ is an associative $k$-algebra spectrum and $k$ is a commutative ring spectrum. We will also generally assume $k$ and $A$ are themselves $K(n)$-local (as otherwise, we get equivalent categories by replacing them with their $K(n)$-localizations).

\begin{prop}
For $M \in \LMod(A)$, the following are equivalent:
\begin{enumerate}
\item $M \in L_{K(n)} \LMod(A)$
\item The $k$-module underlying $M$ is in $L_{K(n)}\Mod(k)$.
\end{enumerate}
\end{prop}

\begin{proof}
The actions of $K(n) \otimes_{\bS}$ on $\LMod(A)$ and on $\Mod(k)$ are intertwined by the forgetful functor, so if $N$ is $K(n)$-acyclic as an $A$-module, it is also $K(n)$-acyclic as a $k$-module. Similarly $A \otimes_k N$, $A \otimes_k A \otimes_k N$,\dots are $K(n)$-acyclic $k$-modules if $N$ is. Thus, for all $n$, $\Maps_k(A^{\otimes_k n} \otimes_k N,M) = 0$ when $M$ is $K(n)$-local and $N$ is $K(n)$-acyclic. We conclude that (2) implies (1) from the bar model for $\Maps_A(N,M)$, i.e. 
\[
\begin{array}{rcl}
\Maps_A(N,M) & \cong & \varprojlim_n \Maps_k(A^{\otimes n} \otimes N,M) \\
 & = & \varprojlim(\Maps_k(
A \otimes_k N,M) \rightrightarrows \Maps_k(A \otimes_k A \otimes_k N,M)  \cdots)
\end{array}
\]
which vanishes as long as all the terms in the limit do, i.e. as long as $M$ is $K(n)$-local as a $k$-module.

To show that (1) implies (2), use again that $A \otimes_k N$ is $K(n)$-acyclic if $N$ is, and that $\Maps_k(N,M) \cong \Maps_A(A \otimes_k N,M)$.

\end{proof}

\subsection{Finiteness conditions for $K(n)$-local modules}

\begin{rmk}\label{rmk:Knfiniteness}
There are analogous finiteness conditions to \S\ref{subsec:finite-k} in $L_{K(n)} \Mod(k)$, but they do not necessarily coincide with the corresponding notions in the larger category $\Mod(k)$, and they no longer all agree.
\begin{enumerate}
\item The notion of perfect is unchanged (any perfect $k$-module is automatically $K(n)$-local).  
\item Compact (or perfect) objects in $\Mod(k)$ are not in general compact in $L_{K(n)}\Mod(k)$; for example $k$ itself is usually not compact in $L_{K(n)} \Mod(k)$.  This is analogous to the fact that $\bZ_p$ is not compact in $p$-complete $\bZ_p$-modules.  
\item The natural tensor product on $L_{K(n)}\Mod(k)$ is not $(M,N) \mapsto M \otimes_k N$ but $(M,N) \mapsto L_{K(n)}(M \otimes_k N)$.  We will call objects which are dualizable with respect to this tensor structure \emph{$K(n)$-locally dualizable}.   Perfect modules are always $K(n)$-locally dualizable, but there are generally more $K(n)$-locally dualizable objects than this \cite[\S 15]{HovStrick}.
\end{enumerate}
\end{rmk}

While the categories $\LModft{A}$ are defined in terms of the underlying $k$-modules being perfect, it turns out that the notion of $K(n)$-local dualizability has more useful formal properties (cf. also \Cref{prop:ambi}):

\begin{lem}\label{lem:Kndbl}
Let $k$ be a $K(n)$-local commutative ring spectrum and $M\in L_{K(n)}\Mod(k)$.  Then $M$ is $K(n)$-locally dualizable if and only if the functor
\[
L_{K(n)}(-\otimes_k M) : L_{K(n)}\Mod(k)\to L_{K(n)}\Mod(k)
\]
preserves limits.
\end{lem}
\begin{proof}
The only if direction follows from $L_{K(n)}\Mod(k)$ being a closed symmetric monoidal category.  For the if direction, note that the adjoint functor theorem \cite[Cor. 5.5.2.9]{HTT} and \Cref{prop:colimpres} imply that $-\otimes_k M$ has a left adjoint given by $L_{K(n)}(-\otimes_k N)$, and $N$ is easily seen to be the $K(n)$-local dual of $M$.  
\end{proof}

In light of \Cref{rmk:Knfiniteness}(3), we highlight a condition where this distinction disappears:

\begin{equation}
\label{dfn:good}
\begin{array}{l}
\text{
A $K(n)$-local commutative ring spectrum $k$ satisfies condition}\\
\text{\eqref{dfn:good} if every $K(n)$-locally dualizable $k$-module $M$ is perfect. }
\end{array}
\end{equation}

Since perfect $k$-modules are always dualizable, the classes of dualizable and perfect $K(n)$-local $k$-modules coincide when $k$ obeys this condition.  Hence, $\LModft{A}$ coincides with the subcategory of $L_{K(n)}\LMod(A)$ whose underlying $k$-module is $K(n)$-locally dualizable.  

\begin{exm}
Condition \eqref{dfn:good} is satisfied when $K(n) = K(1)$ and $k = \KU_p$. In fact, for any $n$, any \emph{Lubin-Tate theory} (equivalently, Morava $E$-theory) obeys the condition.  To see this, let $E$ be a Lubin-Tate theory, let $I = (p,v_1,\cdots ,v_{n-1})\subset \pi_0(E)$ denote a Landweber ideal in $\pi_0(E)$, and let $K = E/(p, v_1, \cdots ,v_{n-1})$ denote the corresponding Morava $K$-theory.  If $M$ is a dualizable $E$-module, then $\pi_*(M/(p, v_1, \cdots ,v_{n-1}))$ is a dualizable $\pi_*(K)$-module, and therefore it is a finitely generated $\pi_*(K)$-module.   By the proof of \cite[Prop. 2.4]{HovStrick}, this means that $\pi_*(M)$ is finitely generated over $\pi_*(E)$, and hence by \cite[Lem. 8.11]{HovStrick} that $M$ is perfect as an $E$-module.  
\end{exm}

\subsection{Functors and bimodules in $L_{K(n)}$-local categories}
\label{subsec:Knlocalfun}

\begin{prop}\label{prop:colimpres}
Let $k$ be a $K(n)$-local commutative ring spectrum and let $A$ and $B$ be $K(n)$-local associative $k$-algebra spectra. 
If $F$ is a colimit-preserving $k$-linear functor
\[
F:L_{K(n)}\LMod(A) \to L_{K(n)}\LMod(B),
\]
then $F$ is isomorphic to $L_{K(n)}(F(A) \otimes_A -)$.
\end{prop}

Here $F(A)$ gets a $(B,A)$-bimodule structure as in \S\ref{subsec:fin-con-fun}. The analog of \eqref{eq:role-of-bimodules} is
\[
\Fun^{\mathrm{L}}_k(L_{K(n)} \LMod(A),L_{K(n)}\LMod(B)) \cong L_{K(n)} \Bimod(B,A)
\]
The functor from the right category to the left category sends a $K(n)$-local bimodule $N$ to the functor $L_{K(n)} (N \otimes_A -)$.

\begin{proof}
Suppose that $F$ is a colimit-preserving $k$-linear functor $L_{K(n)}\LMod(A) \to L_{K(n)} \LMod(B)$. 

For $M \in L_{K(n)} \LMod(A)$ the natural map 
\begin{equation}
\label{eq:before-Kn}
F(A) \otimes_A M \to F(M)
\end{equation}
factors through a natural map
\begin{equation}
\label{eq:after-Kn}
L_{K(n)} (F(A) \otimes_A M) \to F(M).
\end{equation}
Indeed, since $F(M)$ is $K(n)$-local, \eqref{eq:after-Kn} is $L_{K(n)}$ applied to \eqref{eq:before-Kn}. The map \eqref{eq:after-Kn} is an isomorphism when $M = A$, and since $F$ preserves colimits is therefore an isomorphism for all $M \in L_{K(n)} \LMod(A)$.
\end{proof}

When $A$ is perfect over $k$, we have containments
\[
\LMod^{\omega}(A) \subset \LModft{A} \subset L_{K(n)} \LMod(A).
\]
As in \Cref{rmk:Knfiniteness}, we warn that neither category agrees with compact objects in $L_{K(n)}\LMod(A)$.  If $A$ and $B$ are both perfect over $k$, it is natural to ask for a criterion for $L_{K(n)}(F(A) \otimes_A -)$ to carry $\LModft{A}$ into $\LModft{B}$. The criterion of Prop. \ref{prop:FA-right-perfect} still applies, but is strictly stronger than necessary. When $A$ and $B$ are group algebras, a phenomenon called $K(n)$-local \emph{Tate vanishing} or \emph{ambidexterity} gives a less restrictive criterion.  

The study of these phenomena originates in work of Greenlees--Sadofsky and Hovey--Sadofsky \cite{GS, HovSad}.  Generalizing their work, Kuhn \cite{Kuhn} showed that the Tate cohomology of finite groups vanishes in the $K(n)$-local setting\footnote{In fact, he showed this in the (conjecturally) more general setting of $T(n)$-local homotopy theory.}.  The $n=0$ case of Kuhn's theorem is the familiar fact that the additive transfer map from orbits to fixed points is an isomorphism when working over $\mathbf{Q}$.  These results were further generalized and reinterpreted by the Hopkins--Lurie theory of ambidexterity \cite{HL}, which has been generalized and developed extensively by Carmeli--Schlank--Yanovski and Barthel--Carmeli--Schlank--Yanovski \cite{CSYTele}, \cite{CSYHeight}, \cite{CSYCyc}, \cite{BCSY}.
For the convenience of the reader, the following proposition collects some basic features of the theory.

\begin{prop}\label{prop:ambi}
Let $G$ be a finite group and let $k$ be a $K(n)$-local commutative ring spectrum.  Then:
\begin{enumerate}
\item The functors $\Fun(BG, L_{K(n)}\Mod(k))  \to L_{K(n)}\Mod(k)$ given by $(-)^{hG}$ and $L_{K(n)}(-)_{hG}$ are identified.
\item The $k$-modules $k^{BG}$ and $L_{K(n)}k[BG]$ (which are isomorphic by (1)) are $K(n)$-locally dualizable over $k$.
\end{enumerate}
If $G$ is a $p$-group, then we also have
\begin{enumerate}\setcounter{enumi}{2}
\item There is an equivalence of categories 
\[
\Fun(BG, L_{K(n)}\Mod(k)) \cong L_{K(n)}\Mod(k^{BG})
\] 
given by $M\mapsto M^{hG}$. 
\item Under the equivalence in (3), both functors in (1) are identified with restriction of scalars along $k\to k^{BG}$.  
\end{enumerate}
\end{prop}
\begin{proof}
Statement (1) is \cite[Thm 1.1]{HovSad} (or \cite[Thm. 5.2.1]{HL}) and (3) and (4) are \cite[Thm. 5.4.3]{HL}.  To see (2), note first that if $P\subset G$ is a $p$-Sylow, then a standard transfer argument shows that $L_{K(n)}k[BG]$ is a summand of $L_{K(n)}k[BP]$ and so it suffices to consider the case when $G$ is a $p$-group.  The transitivity of homotopy orbits then reduces to the case $G=C_p$, and by base-change, it suffices to consider the case $k= L_{K(n)}\bS$.  Finally, by \cite[Thm 8.6]{HovStrick}, the $K(n)$-local dualizability of $L_{K(n)}\bS[C_p]$ amounts to seeing $K(n)_*(BC_p)$ is finite, which follows from \cite{RavWil}.  
\end{proof}

\begin{prop}\label{prop:preservedbl}
Let $k$ be a $K(n)$-local commutative ring spectrum, let $A = k[G]$ be the group algebra (\S\ref{sub:permmod}) of a finite group $G$, and let $F(A)$ be a right $A$-module whose underlying $k$-module is $K(n)$-locally dualizable.  Then the functor
\begin{align*}
F: L_{K(n)}\LMod(A) &\to L_{K(n)}\Mod(k)\\
M &\mapsto L_{K(n)}(F(A)\otimes_A M)
\end{align*}
preserves the property of having $K(n)$-locally dualizable underlying $k$-module.  
\end{prop}
\begin{proof}
In light of \Cref{prop:ambi}(2-4), this is a consequence of the following lemma \ref{lem:dbldbl} applied to the map $k\to k^{BG}$.
\end{proof}

\begin{lem}\label{lem:dbldbl}
Let $f: k \to k'$ be a map of $K(n)$-local commutative ring spectra such that $k'$ is $K(n)$-locally dualizable as a $k$-module.  Then restriction of scalars along $f$ sends $K(n)$-locally dualizable $k'$-modules to $K(n)$-locally dualizable $k$-modules.  
\end{lem}
\begin{proof}
Let $M$ be a $K(n)$-locally dualizable $k'$-module.  Then note that the functor\footnote{Here and in the following equation, we need not $K(n)$-localize the tensor product due to the assumed dualizability.}
\begin{equation}\label{eqn:tensorM}
L_{K(n)}(-\otimes_k M) : L_{K(n)}\Mod(k) \to L_{K(n)}\Mod(k)
\end{equation}
can be written as the composite
\[
L_{K(n)}\Mod(k) \xrightarrow{-\otimes_k k'} L_{K(n)}\Mod(k') \xrightarrow{-\otimes_{k'} M} L_{K(n)}\Mod(k') \xrightarrow{\text{forget}} L_{K(n)}\Mod(k).
\]
By the dualizability hypotheses and \Cref{lem:Kndbl}, the first two arrows preserve limits, and the third arrow preserves limits as it is a right adjoint.  Thus, (\ref{eqn:tensorM}) preserves limits as well and the claim follows from \Cref{lem:Kndbl}.  
\end{proof}

\begin{rmk}
\Cref{lem:dbldbl} is not specific to $K(n)$-local spectra --- the analogous statement holds more generally in any presentable symmetric monoidal stable $\infty$-category.  
\end{rmk}

\begin{rmk}
The conclusion of \Cref{prop:preservedbl} can fail when $A$ is not the group algebra of a finite group. For example, let $A = k[S^1]$ be the $k$-homology spectrum of the circle.  Then we have an isomorphism
\[
L_{K(n)}(k\otimes_A k) \cong k[BS^1],
\]
which is not generally $K(n)$-locally dualizable.
\end{rmk}

When $k$ satisfies the condition \eqref{dfn:good}, then perfect modules and $K(n)$-locally dualizable modules coincide, and therefore we immediately deduce the following consequence of \Cref{prop:preservedbl} for finite type modules:

\begin{cor}\label{cor:Knfinal}
Let $k$ be a $K(n)$-local commutative ring spectrum satisfying condition \eqref{dfn:good}. Let $A = k[G]$ and $B = k[H]$ be the group algebras (\S\ref{sub:permmod}) of finite groups $G$ and $H$. Let $F(A)$ be a $(B,A)$-bimodule whose underlying $k$-module is perfect. Then the functor
\[
M \mapsto L_{K(n)}(F(A) \otimes_A M)
\]
carries $\LModft{A}$ to $\LModft{B}$.  
\end{cor}

\section{Permutation modules and Rouquier's equivalence over $\bS$}\label{sec:rouquier}
We now turn to lifting Rouquier's equivalence to $\bS$.  
In \S\ref{sub:rouquiercx}, we start by reviewing the form of Rouquier's original equivalence.  Then in \S\ref{sub:permmod}-\ref{sec:strong-segal}, we discuss the theory of permutation $G$-modules over $\bS$.  In \S\ref{sub:permlift}, we show that summands of permutation modules lift from $\bZ_q$ to $\bS_q$.  Finally, in \S\ref{sub:final}, we put these ingredients together with the results of \S\ref{sec:two} to prove our main theorem, \Cref{thm:last}.  

\subsection{Rouquier's two-term complexes}\label{sub:rouquiercx}
Let $G$ be a finite group and let $D \subset G$ be a cyclic group of $p$-power order. The left multiplication of $G$ and the right multiplication of $N_G(D)$ on $G$ endow $\bZ_q[G]$ with a $(G,N_G(D))$-bimodule structure, whose associated functor
\[
\LMod(\bZ_q[N_G(D)]) \to \LMod(\bZ_q[G])
\]
is equivalent to induction along the inclusion $N_G(D) \subset G$. The adjoint is restriction along the same inclusion, represented by $\bZ_q[G]$ regarded as a $(N_G(D),G)$-bimodule.

Now let $\bF_q[G]b$ and $\bF_q[N_G(D)]b'$ be blocks whose defect group is $D$ and which are Brauer correspondents of each other. Then (abusing notation and writing $b$ and $b'$ also for the unique lifts of these idempotents to $\bZ_q[G]$ and $\bZ_q[N_G(D)]$) the $(\bZ_q[G]b,\bZ_q[N_G(D)]b')$-summand of the bimodule $\bZ_q[G]$ is $b\bZ_q[G]b'$, which represents the composite functor
\begin{equation}
\label{eq:ind-block}
\LMod(\bZ_q[N_G(D)]b') \to \LMod(\bZ_q[N_G(D)]) \to  \LMod(\bZ_q[G]) \to \LMod(\bZ_q[G]b)
\end{equation}
This functor carries $\LModft{\bZ_q[N_G(D)]b'}$ into $\LModft{\bZ_q[G]b}$ and $\LMod(\bZ_q[N_G(D)]b')^{\omega}$ into $\LMod(\bZ_q[G]b)^{\omega}$.
Its right adjoint is represented by $b'\bZ_q[G]b$ and carries $\LModft{\bZ_q[G]b}$ into $\LModft{\bZ_q[N_G(D)]b'}$  and $\LMod(\bZ_q[G]b)^{\omega}$ into $\LMod(\bZ_q[N_G(D)]b')^{\omega}$.

The composition \eqref{eq:ind-block} is not an equivalence.  For many finite groups, it does induce a ``stable equivalence'' (for instance, for groups whose $p$-Sylows have trivial intersection \cite[\S 5, 6.4]{Broue}; in general one has to pass to a summand) --- one formulation of this is that it induces an equivalence of quotient categories
\[
\frac{\LModft{\bZ_q[N_G(D)]b'}}{\LMod(\bZ_q[N_G(D)]b')^{\omega}} \cong \frac{\LModft{\bZ_q[G]b}}{\LMod(\bZ_q[G]b)^{\omega}}
\]
In particular the composite of \eqref{eq:ind-block} with its adjoint differs from the identity functor by a bounded complex of projective $(\bZ_q[N_G(D)]b',\bZ_q[N_G(D)]b')$-bimodules or projective $(\bZ_q[G]b,\bZ_q[G]b)$-bimodules.

Rouquier showed that, when $D$ is cyclic, one can introduce a relatively simple correction to the bimodule representing \eqref{eq:ind-block} to obtain an equivalence predicted by Brou\'e's Conjecture.

\begin{thm}[Rouquier \cite{rouquier-cyclic}]
\label{thm:rouquier}
Let $G$ be a finite group and let $D \subset G$ be a cyclic subgroup of $p$-power order. Let $\bF_q[G]b$ and $\bF_q[N_G(D)]b'$ be blocks whose defect group is $D$ and which are Brauer correspondents of each other. Then there is a two-term complex of $(\bZ_q[G]b,\bZ_q[N_G(D)]b')$-bimodules
\[
M_0 = \{\cdots \to 0 \to N_0' \to N_0 \to 0 \to \cdots\}
\]
with the following properties:
\begin{enumerate}
\item $N_0$ is a direct summand of $\bZ_q[G]$, with its $(\bZ_q[G],\bZ_q[N_G(D)])$-bimodule structure coming from the left multiplication action of $G$ and the right multiplication action of $N_G(D)$.
\item $N'_0$ is a projective $(\bZ_q[G]b,\bZ_q[N_G(D)]b')$-bimodule.
\item The resulting functor
\[
M_0 \otimes_{\bZ_q[N_G(D)]b'}-:\LMod(\bZ_q[N_G(D)]b') \to \LMod(\bZ_q[G]b)
\]
is an equivalence, and restricts to an equivalence on $\LMod^{\mathrm{ft}}$ and $\LMod^{\omega}$.
\end{enumerate}
\end{thm}

\subsection{Group rings and permutation modules}\label{sub:permmod}

If $k$ is a commutative ring spectrum, denote by $X \mapsto k[X] \in \Mod(k)$ the functor that carries a space to its $k$-homology spectrum. For example, if $k = \bS$, then $\bS[X]$ is the suspension spectrum of $X$ with a disjoint basepoint and in general $k[X] = \bS[X] \otimes_{\bS} k$.  A useful formula is
\begin{equation}
\label{eq:noneq}
[k[X],\Sigma^i k] \cong H^i(X;k)
\end{equation}
where the right-hand side denotes the $i$th (extraordinary) cohomology of $X$ with coefficients in $k$.

If $G$ acts on $X$, then we may regard $k[X]$ as a $G$-object in $\Mod(k)$ (i.e., as an object of $\Fun(BG,\Mod(k)$)) by composing
\[
BG \to \cS \xrightarrow{k[-]} \Mod(k).
\]
The analog of \eqref{eq:noneq} is
\begin{equation}
\label{eq:equivariant-coh1}
[k[X],\Sigma^i k]_{\Fun(BG,\Mod(k))} \cong H^i_G(X;k) := H^i(X_{hG};k)
\end{equation}
where $X_{hG}:=\varinjlim_{BG} X \cong (X \times EG)/G$ is the Borel construction.

If $X$ is a finite $G$-set, then $k[X] = \bigoplus_{x \in X} k$ is called a permutation module. The associative $k$-algebra structure on $k[G]$ induced by the multiplication on $G$ coincides with the endomorphism algebra of $k[G/\{1\}]$ regarded as a permutation module, which generates $\Fun(BG,\Mod(k))$, so that we have
\[
\Fun(BG,\Mod(k)) \cong \LMod(k[G]).
\]
The symmetric monoidal structure on $\Mod(k)$ induces a symmetric monoidal structure on $\Fun(BG,\Mod(k))$, which we denote by $\otimes_k$. The unit of the monoidal structure is the trivial module $k$. 

\begin{lem}\label{lem: permutation self-dual}
The permutation modules $k[X]$ are self-dual with respect to the monoidal structure in $\Mod(k)$ and in $\Fun(BG,\Mod(k))$.
\end{lem}

\begin{proof} We produce evaluation and coevaluation maps satisfying the necessary relations.

The coevaluation map comes from the composition
\begin{equation}
\label{eq:perm-dual}
k \to k[X] \xrightarrow{\Delta_X} k[X \times X] \cong k[X] \otimes_k k[X]
\end{equation}
where the first map is the diagonal $k \to \bigoplus_{x \in X} k$. The composite
\[
[k[X] \otimes k[X],k] \xrightarrow{\otimes k[X]} [k[X] \otimes k[X] \otimes k[X],k[X]] \xrightarrow{\mathrm{id}_{k[X]} \otimes \eqref{eq:perm-dual}^*} [k[X],k[X]]
\]
is a bijection. The homotopy class of the evaluation map $k[X] \otimes k[X] \to k$ is the image of $\mathrm{id}_{k[X]}$ under the inverse bijection.
\end{proof}

\begin{rmk}
The $k$-homology of manifolds furnish a more general class of dualizable objects of $\Mod(k)$, and the $k$-homology of $G$-manifolds furnish dualizable objects of $\Fun(BG,\Mod(k))$. For a closed $n$-manifold $M$, the homology spectrum $k[M]$ is self-dual up to a shift so long as $M$ is \emph{$k$-orientable} --- indeed one definition of a $k$-orientation of $M$ is a homotopy class of maps $k \to \Sigma^{-n} k[M]$ with the property that the composite map
\[
k \to \Sigma^{-n} k[M] \to \Sigma^{-n} k[M \times M] \to \Sigma^{-n} k[M] \otimes_k k[M]
\]
exhibits $\Sigma^{-n} k[M]$ as the dual of $k[M]$. Lemma \ref{lem: permutation self-dual} is the case $n=0$.
\end{rmk}

Under the equivalence 
\[
\Bimod(k[G],k[G]) \cong \LMod(k[G] \otimes k[G]^{\op}) \cong \LMod(k[G \times G^{\op}]),
\]
the diagonal bimodule $k[G]$ is a permutation $G \times G^{\op}$-module, with $G \times G^{\op}$ acting on $G$ by left and right multiplication. It follows from \eqref{eq:perm-dual} that each $k[G]$ has a natural symmetric structure in the sense of \S\ref{sec: symmetric}. Furthermore any block of $k[G]$ has a symmetric structure: indeed, if $A = A_1 \times A_2 \times \cdots$, then a symmetric structure on $A$ induces a symmetric structure on each $A_i$ by the composite $A_i \to A \cong A^* \to A_i^*$.

The self-duality of permutation modules gives an identification between homotopy classes of maps $k[X] \to k[Y]$ and equivariant $k$-cohomology classes in $X \times Y$:
\[
[k[X],k[Y]]_{k[G]} \cong [k[X] \otimes k[Y],k]_{k[G]} \cong [k[X \times Y],k]_{k[G]} \cong H^0_G(X \times Y;k).
\]

\subsection{The Burnside ring}\label{sub:burn}
For a finite group $G$ and a finite $G$-set $X$, let $\Burn_G(X)$ denote the Grothendieck group of the commutative monoid whose objects are isomorphism classes of finite $G$-sets $Y$ equipped with a $G$-equivariant map to $X$, and whose addition is given by disjoint union. We will refer to $\Burn_G(X)$ as the Burnside ring of virtual finite $G$-sets over $X$, as it acquires a commutative multiplication given by fiber product over $X$.

\begin{rmk}
\label{rmk:burngx}
In general, $\Burn_G(X) = \bigoplus_{x \in G\backslash X} \Burn_{G_x}(\pt)$, where the sum is over $G$-orbit representatives and $G_x$ denotes the stabilizer. In particular, 
$\Burn_G(G/H)$ is isomorphic to $\Burn_H(\pt)$ and $\Burn_G(G/\{1\}) = \bZ$.
\end{rmk}

A map of $G$-sets $X \to X'$ induces by fiber-product over $X'$ a ring homomorphism $\Burn_G(X') \to \Burn_G(X)$. In particular the map $X \to \pt$ gives $\Burn_G(X)$ a $\Burn_G(\pt)$-module structure, and the map $G \to \pt$ induces the \emph{augmentation map} $\Burn_G(\pt) \to \Burn_G(G) = \bZ$, carrying a $G$-set to the cardinality of its underlying set. Let $I \subset \Burn_G(\pt)$ denote the kernel of the augmentation map, called the \emph{augmentation ideal}.

\begin{thm}[{{Carlsson \cite{carlsson}, weak form of Segal's conjecture}}]\label{thm: Carlsson}
Let $G$ be a finite group, which we regard as acting trivially on $\bS$, and let $\bS^{hG}$ be the homotopy fixed-point spectrum of this action \S\ref{subsec:not}. Then there is a natural map of rings
\begin{equation}
\label{eq:segal-conj}
\Burn_G(\pt) \to \pi_0(\bS^{hG})
\end{equation}
which exhibits the target as the $I$-adic completion of the source.
\end{thm}

\begin{rmk}
As the $G$-action on $\bS$ is trivial, the homotopy fixed points $\bS^{hG}$ are naturally identified with $\underline{\Maps}_{\Mod(\bS)}(\bS[BG],\bS)$ and with
$\underline{\Maps}_{\LMod(\bS[G])}(\bS,\bS)$. (Similarly, $\bS_{hG}$ is naturally identified with $\bS[BG]$.)
\end{rmk}

We denote the completion at $I$ of $\Burn_G(\pt)$ by $\Burn_G(\pt)^{\wedge}_I$. If $X$ is a finite $G$-set, we also define $\Burn_G(X)^{\wedge}_I$ using the natural $\Burn_G(\pt)$-module structure on $\Burn_G(X)$.

\begin{prop}
Let $G$ be a finite group.  Then for any two finite $G$-sets $X$ and $Y$, we have an identification
\[
\pi_0 \Maps_{\LMod(\bS[G])}(\bS[X],\bS[Y]) \cong \Burn_G(X \times Y)^{\wedge}_I.
\]
In particular, by taking $Y= \pt$, we have
\begin{equation}\label{eq:burngx}
\pi_0 \Maps_{\LMod(\bS[G])}(\bS[X],\bS) \cong \Burn_G(X)^{\wedge}_I.
\end{equation}
\end{prop}

\begin{proof}
In fact, it suffices to prove the special case (\ref{eq:burngx}) noted in the statement, as $\bS[Y]$ is canonically self-dual and therefore
\[
\Maps_G(\bS[X],\bS[Y]) \cong \Maps_G(\bS[X] \otimes \bS[Y],\bS) \cong \Maps_G(\bS[X \times Y],\bS).
\]
For this, note that both sides of \eqref{eq:burngx} convert disjoint union into products, so it's enough to consider the case where $X$ has a single orbit, say $X = G/H$. Then 
\[
\Maps_G(\bS[G/H],\bS) \cong \Maps_H(\bS,\bS),
\]
whose $\pi_0$ is identified by Theorem \ref{thm: Carlsson} with $\Burn_H(\pt)_{I_H}^\wedge \cong \Burn_G(G/H)_{I_G}^\wedge$.
\end{proof}

If we further complete $\Burn_G(X)^{\wedge}_I$ at $p$, or equivalently if we study $\Burn_G(X)^{\wedge}_{(I,p)}$ where $(I,p)$ is the kernel of $\Burn_G(X) \xrightarrow{\mathrm{aug}} \bZ \to \bZ/p$, we have the following simple description:

\begin{lem}\label{lem:fgZp}
The $\bZ_p$-module $\Burn_G(\pt)^{\wedge}_{(I,p)}$ is free of finite rank, generated by isomorphism classes of $G$-sets of the form $G/P$ for $P \subset G$ a $p$-group.
\end{lem}

Because of this and Remark \ref{rmk:burngx}, $\Burn_G(X)^{\wedge}_{(I,p)}$ is a free $\bZ_p$-module generated by the set of isomorphism classes of $G$-maps $G/P \to X$, where $P\subset G$ is a $p$-group.

\begin{proof}
Recall Burnside's marking homomorphism $\Burn_G(\pt) \to \prod_H \bZ$, also called the ``table of marks'' \cite{burnside}. The factors in the direct product are indexed by conjugacy class representatives of subgroups of $G$, and the projection onto the factor $H$ sends the $G$-set $X$ to the cardinality of its $H$-fixed points. If $s \in \Burn_G(\pt)$ we will write $\#s^H$ for the composition of the marking homomorphism with the projection onto the $\bZ$ factor indexed by $H$.

The marking homomorphism is a ring homomorphism, and it is easy to see that it is injective (a reference is \cite[Proposition 1.2.2]{td79}).  Since its domain and codomain are free abelian groups of the same rank, this means it becomes an isomorphism after applying $- \otimes {\bQ}$ or $- \otimes {\bQ_p}$. Write $e_H \in \Burn_G(\pt) \otimes \bQ$ for the virtual $G$-set with one $H$-fixed point and no $H'$-fixed points when $H'$ is not conjugate to $H$:
\[
\#e_H^H = 1 \qquad \#e_H^{H'} = 0 \text{ when $H'$ is not conjugate to $H$}
\]
The $e_H$ are the primitive idempotents in $\Burn_G(\pt) \otimes \bQ$ and in $\Burn_G(\pt) \otimes \bQ_p$.  We remark that $e_H$ only involves $G$-sets corresponding to subconjugates of $H$: to see this, consider the variant of the marking homomorphism where on the source, one takes the subring generated by $G$-sets subconjugate to $H$, and on the target, one considers only factors corresponding to subconjugates of $H$. This variant is easily seen to be injective again, and is therefore an isomorphism after tensoring with $\Q$. 

The primitive idempotents in $\Burn_G(\pt) \otimes \bZ_p \cong \Burn_G(\pt)^{\wedge}_{p}$ were identified by Dress \cite{dress}. They are in one-to-one correspondence with conjugacy classes of $p$-perfect subgroups $\varpi \subset G$ (a group is called $p$-perfect when it has no normal subgroups of $p$-power index). The idempotent corresponding to $\varpi$ is
$\varepsilon_{\varpi} = \sum_{H} e_H$, where the sum runs through class representatives of subgroups $H$ that contain a conjugate of $\varpi$ as a normal subgroup of $p$-power index.

The trivial subgroup is $p$-perfect; let us denote its corresponding idempotent by $\varepsilon := \varepsilon_1$. Thus
$
\varepsilon := \varepsilon_1 = \sum_Q e_Q
$
where the sum runs through class representatives of $p$-subgroups $Q \subset G$.
The natural map
\begin{equation}\label{eq: burn 1}
\Burn_G(\pt)^{\wedge}_p \to \Burn_G(\pt)^{\wedge}_{(I,p)}
\end{equation}
kills all the other $\varepsilon_{\varpi}$ and has target a local ring (the target is the completion of $\Burn_G(\pt)$ at a maximal ideal, namely the kernel of augmentation modulo $p$).  Therefore, \eqref{eq: burn 1} factors through 
\begin{equation}\label{eq: burn 2}
\varepsilon \Burn_G(\pt)^{\wedge}_p \to \Burn_G(\pt)^{\wedge}_{(I,p)}
\end{equation}
Since the source ring $\varepsilon \Burn_G(\pt)^{\wedge}_p$ is local and (by the injectivity of the marking homomorphism) finitely generated and free over $\Z_p$, the $(I,p)$-adic topology on it coincides with the $p$-adic topology: $\varepsilon \Burn_G(\pt)^{\wedge}_p/(p)$ is a finite local ring so the image of the ideal $(I,p)$ is nilpotent. Therefore, $\varepsilon \Burn_G(\pt)^{\wedge}_p$ is isomorphic to its own completion with respect to the ideal generated by $(I,p)$, so that \eqref{eq: burn 2} is an isomorphism.

Furthermore, note that for any $p$-subgroup $Q\subset G$, we have $\varepsilon G/Q = G/Q$.  This can be seen from the fact that they have the same image under the marking homomorphism, as $(G/Q)^H$ is nonempty if and only if $H$ is conjugate to a subgroup of $Q$, and $\varepsilon$ is the indicator function on the factors indexed by $p$-groups $H$. 

It therefore suffices to show that the $G$-sets of the form $\varepsilon G/Q$ span $\varepsilon\Burn_G(\pt)^{\wedge}_p$.  But the $G$-sets of the form $G/H$ span $\Burn_G(\pt)$ and inside $\varepsilon \Burn_G(\pt)^{\wedge}_p$, we have $G/H = \varepsilon G/H$.  As $\varepsilon = \sum_Q e_Q$ only involves $G$-sets of the form $G/Q$ for $p$-subgroups $Q$ and $G/H \times G/Q$ involves only $G$-sets with isotropy subconjugate to $Q$, we conclude.  
\end{proof}

\subsection{Digression on the strong Segal conjecture and Tate cohomology} \label{sec:strong-segal}
This subsection is not used in the proof of \S\ref{intro:maintheorem}.

Segal formulated \eqref{eq:segal-conj} by analogy with Atiyah's theorem \cite{atiyah}, which identifies the complex $K$-theory $\KU^0(BG)$ with a completion of the representation ring $R(G)$ of $G$, i.e. the Grothendieck ring of finite-dimensional representations of $\bC[G]$:
\begin{equation}
\label{eqn:ascomp}
R(G) \to \KU^0(BG).
\end{equation}

There is a more sophisticated version of both Atiyah's theorem and of the Segal conjecture, which in a way is the start of ``genuine'' equivariant stable homotopy theory.
\begin{enumerate}
\item There is a commutative ring spectrum $\KU^G$, the (genuine) $G$-equivariant $K$-theory of a point, together with a map
\[
\KU^G \to \KU^{hG}
\]
which recovers \eqref{eqn:ascomp} on $\pi_0$. It is constructed out of the category of complex representations of $G$ in about the same way that $\KU$ is constructed out of the category of complex vector spaces. As an object of $\Mod(\bS)$ or $\Mod(\KU)$, $\KU^G$ has a simple structure: it is a free $\KU$-module spanned by the irreducible complex representations:
\[
\KU^G \cong \bigoplus_{\mathrm{Irr}(\bC[G])} \KU.
\]

\item Let $\bS^G$ denote the group completion $K$-theory spectrum of the category of finite $G$-sets, with its symmetric monoidal structure given by disjoint union of $G$-sets. The Cartesian product of $G$-sets endows $\bS^G$ with the structure of a commutative ring spectrum.
As an object of $\Mod(\bS)$ one has
\[
\bS^G \cong \bigoplus_H \bS[B(N_G(H)/H)]
\] 
where the sum ranges over conjugacy class representatives of subgroups $H \subset G$. 
\end{enumerate}

Carlsson proved more than \eqref{eq:segal-conj}:
\begin{thm}[{\cite{carlsson}, strong form of Segal's conjecture}]
Let $G$ be a finite group. There is a natural map of commutative ring spectra
\[
\bS^G \to \bS^{hG}
\]
which exhibits $\pi_*(\bS^{hG})$ as the $I$-adic completion of $\pi_*(\bS^G)$.
\end{thm}

This has a surprising consequence: $\bS^{hG}$ is connective. (To see that it is surprising, note that replacing $\bS$ with the Eilenberg-MacLane spectrum of $\bZ$ yields an object concentrated in non-positive degrees, as $\pi_i(\bZ^{hG})$ is the $(-i)$th group cohomology of $G$).

A second consequence of the Segal conjecture is that one can compute the Tate cohomology of some finite groups with coefficients in the sphere. For a finite group $G$ and a $G$-spectrum $M$, the $G$-Tate cohomology $M^{tG} \in \Mod(\bS)$ of $M$ is defined via a cofiber sequence
\[
M_{hG} \xrightarrow{\Nm} M^{hG} \to M^{tG}
\]
When $M$ is discrete, the map $\Nm$ is given on $\pi_0$ by the formula $m \mapsto \sum_{g \in G} gm$, and the homotopy groups $\pi_* M^{tG} = \hat{H}^{-*}(G;M)$ coincide with the classical Tate cohomology groups of $G$ with coefficients in $M$.

\begin{exm}
When $G = C_p$, the cyclic group of order $p$, the nature of the augmentation ideal is that the strong Segal conjecture identifies $\bS^{hC_p}$ as $\bS_p \oplus \bS[BC_p]$. The norm map is the inclusion of the second summand, $\bS[BC_p] = \bS_{hC_p}$. It follows that $\bS^{tC_p} \cong \bS_p$. It is in this form that the Segal conjecture for $C_p$ was originally proved by Lin ($p = 2$, \cite{Lin}) and Gunawardena ($p$ odd, \cite{gunawardena}).
\end{exm}

\subsection{$p$-permutation $\bS_q$-modules}\label{sub:permlift}
We say that a $\bZ_q[G]$-module $M_0$ \emph{lifts to $\bS_q$} if there exists $M \in \LMod(\bS_q[G])$ such that there is an equivalence 
\[
M \otimes_{\bS_q} \bZ_q \cong M_0
\]
 of $\bZ_q[G]$-modules. The question of whether a $\bZ_q[G]$-module lifts to $\bS_q$ is a close relative of the ``equivariant Moore space problem,'' and negative examples were first given by Carlsson in the cases $G = C_p\times C_p$ \cite{carlsson2}. In this section we discuss a class of modules for which a natural lift to the sphere does exist: for summands of permutation modules, by virtue of the following proposition:

\begin{prop}
\label{prop:Sp-permutation}
Let $G$ be a finite group, let $X$ be a finite $G$-set, and suppose that $\bZ_q[X] \cong M_0 \oplus N_0$ is a $\bZ_q[G]$-module splitting. Then, there exists a $\bS_q[G]$-module splitting $\bS_q[X] \cong M \oplus N$ which recovers the $\bZ_q[G]$-module splitting above after applying $\otimes_{\bS_q} \bZ_q$.
\end{prop}

The map of ring spectra $\bS_q \to \bZ_q$, given by truncation, induces a map of endomorphism rings
\begin{equation}
\label{eq:lift-idems-here}
\pi_0\Maps_{\bS_q[G]}(\bS_q[X],\bS_q[X]) \to \pi_0\Maps_{\bZ_q[G]}(\bZ_q[X],\bZ_q[X])
\end{equation}
and the content of the Proposition is that any primitive idempotent in the target lifts to a primitive idempotent in the source. By \S \ref{sub:burn}, this map can be identified with the natural map
\[
\Burn_G(X\times X)^{\wedge}_{(p,I)}\otimes_{\bZ_p}\bZ_q \to \pi_0\Maps_{\bZ_q[G]}(\bZ_q[X\times X],\bZ_q)
\]
which sends a $G$-set over $X\times X$ to its rank function.  In particular, the source (\Cref{lem:fgZp}) and target are free $\bZ_q$-modules of finite rank, and the map is a surjection of (not necessarily commutative) $\bZ_q$-algebras, so this is a consequence of the following purely algebraic fact:

\begin{prop}
\label{prop:Sp-permutation2}
Let $ \bar{A}$ and $\bar{B}$ be associative $\bZ_q$-algebras that are finitely generated as $\bZ_q$-modules. Let $f:\bar{A} \to \bar{B}$ be a surjective algebra homomorphism. Then any primitive idempotent of $\bar{B}$ lifts to a primitive idempotent of $\bar{A}$.
\end{prop}

We put bars over these $\bZ_q$-algebras to emphasize that they are ordinary abelian groups, not ring spectra. It is likely that this is an old result, but we didn't find any place that it was stated very plainly, so we supply a proof in the rest of this section.

\begin{lem}
Let $\bar{A}$ and $\bar{B}$ be finitely generated associative algebras over a field, and let $f:\bar{A} \to \bar{B}$ be a surjective algebra homomorphism. Then every invertible element of $\bar{B}$ has an invertible preimage under $f$.
\end{lem}

\begin{proof}
Let $I \subset \bar{A}$ be the kernel of $f$, so that $f$ is isomorphic to the quotient map $\bar{A} \to \bar{A}/I$. We will show that every invertible element of $\bar{A}/I$ lifts to an invertible element of $\bar{A}$.

Let $J$ be the Jacobson radical of $\bar{A}$. Let us first prove the Lemma in the case that $I \cap J = 0$. If $I \cap J = 0$, then the natural map $\bar{A} \to \bar{A}/I \times_{\bar{A}/I+J} \bar{A}/J$ is an isomorphism of algebras, thus for every unit $u + I \in \bar{A}/I$, we seek a unit $v+J \in \bar{A}/J$ such that $u+I+J = v+I+J$. Such a $v+J$ exists by Wedderburn's theorem: $\bar{A}/J$ is semisimple thus a product of a finite set of simple algebras, and the image of $I$ in $\bar{A}/J$ is a product of a subset of the same simple algebras.

In case $I \cap J \neq 0$, the above argument shows we can lift units along $A/(I \cap J) \to A/I$. Since $\bar{A}$ is finite-dimensional over a field, it is an Artinian ring so $J$ (and therefore $I \cap J$) is nilpotent \cite[Theorem 1.3.1]{herstein}. It follows that units lift along $A \to A/I\cap J$.
\end{proof}

\begin{lem}
\label{lem:inv-lift}
Let $\bar{A}$ and $\bar{B}$ be associative $\bZ_p$-algebras that are finitely generated as $\bZ_p$-modules. Let $f:\bar{A} \twoheadrightarrow \bar{B}$ be a surjective algebra homomorphism. Then every invertible element of $\bar{B}$ has an invertible preimage under $f$.
\end{lem}

\begin{proof}
Suppose $b \in \bar{B}$ is invertible. Let us first show there is an $a \in \bar{A}$ such that 
\begin{enumerate}
\item $a+px$ is invertible for all $x \in A$
\item $f(a) = b$ mod $p\bar{B}$
\end{enumerate}
Indeed $b + p\bar{B}$ is invertible in the $\bF_p$-algebra $\bar{B}/p\bar{B}$, and by applying the previous Lemma to $\bar{A}/p\bar{A} \to \bar{B}/p\bar{B}$ we conclude there is an $a \in \bar{A}$ such that $a + p\bar{A}$ is invertible in $A/p\bar{A}$ and $f(a) + p\bar{B} = b + p\bar{B}$. Since $\bar{A}$ is finitely generated over $\bZ_p$, it is $p$-adically complete and such an $a$ is invertible in $\bar{A}$: if $aa' = 1 + p \epsilon$ then $a'(1-p\epsilon +p^2 \epsilon^2 - \cdots)$ is the inverse to $a$.

It remains to find $x$ such that $f(a+px) = b$. By (2) above, $f(a) - b = py$ for some $y \in \bar{B}$. Let $x$ be such that $f(x) = -y$. Then
\[
f(a + px)-b = f(a) + pf(x) - b = f(a)-b + pf(x) = py +pf(x) = 0
\]
\end{proof}

\begin{lem}
\label{lem:conj-idem}
Let $\bar{A}$ be a $\bZ_p$-algebra that is finitely generated as a $\bZ_p$-module. Let $i,j \in \bar{A}$ be two idempotent elements. Then $\bar{A}i \cong \bar{A}j$ as left $\bar{A}$-modules if and only if $i$ is conjugate to $j$.
\end{lem}

\begin{proof}
If $i = uju^{-1}$, then right multiplication by $u$ induces an isomorphism
$
\bar{A}i = \bar{A}uju^{-1} \stackrel{\sim}{\to} \bar{A}j
$.
Let us show the converse. Suppose $\phi:\bar{A}i \stackrel{\sim}{\to} \bar{A}j$ is an isomorphism of left $\bar{A}$-modules. Since $\phi(ai) = a\phi(i)$, $\phi$ is determined by $\phi(i) \in \bar{A}j$. Furthermore, since $i^2 = i$, $\phi(i) = \phi(i^2) = i\phi(i)$, so $\phi(i) \in i\bar{A}j$. Similarly, $\phi^{-1}$ is determined by $\phi^{-1}(j) \in j\bar{A}i$. The computations
\[
\begin{array}{c}
\phi^{-1}(j) \phi(i) = \phi(\phi^{-1}(j)i) = \phi(\phi^{-1}(j)) = j\\
\phi(i) \phi^{-1}(j) = \phi^{-1}(\phi(i)j) = \phi^{-1}(\phi(i)) = i
\end{array}
\]
show that $\phi^{-1}(j) \phi(i) = j$ and $\phi(i) \phi^{-1}(j) = i$

Recall that $\bZ_p$-algebras that are finitely generated as $\bZ_p$-modules are ``semiperfect'' in the sense of \cite[\S 2.1]{bass}. This means that finitely generated $\bar{A}$-modules --- such as $\bar{A}$ itself --- have the Krull-Schmidt property, a unique direct sum decomposition into indecomposable summands. Since $\bar{A} = \bar{A}i \oplus \bar{A}(1-i) = \bar{A}j \oplus \bar{A}(1-j)$, the hypotheses of the Lemma also show that $\bar{A}(1-i) \cong \bar{A}(1-j)$, via an isomorphism $\psi$. Thus by the same argument as above, $\psi(1-i) \in (1-i)A(1-j)$, $\psi^{-1}(1-j) \in (1-j) A (1-i)$, and
\[
\psi^{-1}(1-j)\psi(1-i) = 1-i \qquad \psi(1-i)\psi^{-1}(1-j) = 1-j.
\]

Let $u = \phi(i) + \psi(1-i)$ and $v = \phi^{-1}(j) + \psi^{-1}(1-j)$. Then $uv = 1=vu$, so $v = u^{-1}$, and $u$ conjugates $j$ to $i$:
\[
\begin{array}{rcl}
uju^{-1} & = & (\phi(i) + \psi(1-i))j(\phi^{-1}(j)+\psi^{-1}(1-j)) \\
& = & (\phi(i)j + \psi(1-i)j)(\phi^{-1}(j)+\psi^{-1}(1-j)) \\
& = & (\phi(i) + 0)(\phi^{-1}(j)+\psi^{-1}(1-j))
\\
& = & \phi(i)\phi^{-1}(j) + \phi(i)\psi^{-1}(1-j) \\
& = & \phi(i)\phi^{-1}(j) + 0 \\
& = & i
\end{array}
\]
\end{proof}

\begin{proof}[Proof of Proposition \ref{prop:Sp-permutation}]
It suffices to prove Proposition \ref{prop:Sp-permutation2}.
Choose a decomposition $1_{\bar{A}} = \sum e_i$, where the $e_i$ are orthogonal idempotents. Then $1_{\bar{B}} = f(1_{\bar{A}}) = \sum f(e_i)$. The set of nonzero $f(e_i)$ is a set of orthogonal idempotents in $\bar{B}$ that sum to $1$. 

Now suppose the $e_i$ are all primitive, and let us show that the nonzero $f(e_i)$ are also primitive. Each $\bar{A}e_i$ surjects onto $\bar{B}f(e_i)$; since this is a map of $\bar{A}$-modules and $\bar{A}e_i$ is an indecomposable projective $\bar{A}$-module, it follows that $\bar{B}f(e_i)$ is indecomposable as an $\bar{A}$-module.

Since the action of $\bar{A}$ factors through the surjection to $\bar{B}$, $\bar{B}f(e_i)$ is also indecomposable as a $\bar{B}$-module. Thus each $f(e_i)$ is primitive.

 If $e'$ is some other primitive idempotent of $\bar{B}$, then $\bar{B} \cong \bar{B}e' \oplus \bar{B}(1-e') \cong \bigoplus \bar{B}f(e)i$, so $Be' \cong Bf(e_i)$ for some $i$ by the Krull-Schmidt property. By Lemma \ref{lem:conj-idem}, there is an invertible $u \in B$ such that $uf(e_i)u^{-1} = e'$, and by Lemma \ref{lem:inv-lift} there is a unit $u' \in \bar{A}$ such that $f(u') = u$. Then $u' e_i (u')^{-1}$ is a primitive idempotent of $\bar{A}$ that lifts $e'$.
\end{proof}

\subsection{Blocks of $\bS_q[G]$}
\label{subsec:impliesbroue}

Let $b_1,b_2,\ldots$ be the block idempotents of $\bZ_q[G]$, so that
\begin{equation}
\label{eq:lossiferum}
\bZ_q[G] = \bZ_q[G]b_1 \times \bZ_q[G]b_2 \times \cdots
\end{equation}
By Proposition \ref{prop:barP}, there is a corresponding splitting of $\bS_q[G/\{1\}]$ as a left $G$-module --- write $\bS_q[G/\{1\}]b_i$ for the summand matching $\bZ_q[G]b_i$. Let $\bS_q[G]b_i$ denote the endomorphism $\bS_q$-algebra spectrum of $\bS_q[G/\{1\}]b_i$. By Proposition \ref{prop:PM},  $[\bS_q[G/\{1\}]b_i,\bS_q[G/\{1\}]b_j] = 0$ when $i \neq j$, and therefore there is an isomorphism of $\bS_q$-algebra spectra
\begin{equation}
\label{eq:insideout}
\bS_q[G] = \bS_q[G]b_1 \times \bS_q[G]b_2 \times \cdots
\end{equation}
that induces \eqref{eq:lossiferum} on taking $\pi_0$ or on applying $\otimes_{\bS_q} \bZ_q$.

\begin{prop}\label{prop: finiteness}
\label{prop:impliesbroue}
Let $G$ and $G'$ be finite groups, let $b$ be a block of $\bZ_q[G]$ and let $b'$ be a block of $\bZ_q[G']$. Suppose that there is an equivalence of stable $\infty$-categories 
\begin{equation}
\label{eq:psyntrophicum}
\LMod(\bS_q[G]b) \cong \LMod(\bS_q[G']b').
\end{equation}
Then for any $\bS_q$-algebra spectrum $k$, there are equivalences
\begin{enumerate}
\item $\LMod(k[G]b) \cong \LMod(k[G']b')$
\item $\LModft{k[G]b_i} \cong \LModft{k[G']b'_i}$
\item $\LMod(k[G]b_i)^{\omega} \cong \LMod(k[G]b_i)^{\omega}$
\end{enumerate}
\end{prop}

\begin{proof}
That \eqref{eq:psyntrophicum} implies (1) and (3) is a consequence of \eqref{eq:haukensis}; that it implies (2) is a consequence of Corollary \ref{prop:restrict-to-ft}.
\end{proof}

\subsection{The Rouquier equivalence over $\bS_q$}\label{sub:final}

\begin{prop}
\label{prop:rouq-S_q}
Suppose we are in the situation of Theorem \ref{thm:rouquier}, and let $M_0$ denote the complex of bimodules of that Theorem regarded as an object in $\Bimod(\bZ_q[G]b,\bZ_q[N_G(D)]b')$. Then there is an $M\in\Bimod(\bS_q[G]b,\bS_q[N_G(D)]b')$ such that 
$M \otimes_{\bS_q} \bZ_q$ is isomorphic to $M_0$.
\end{prop}

\begin{proof}
By Proposition \ref{prop:Sp-permutation}, $N_0$ lifts to a $(\bS_q[G]b,\bS_q[N_G(D)]b')$-bimodule $N$ such that $N \otimes_{\bS_q} \bZ_q \cong N_0$. Since $N_0'$ is projective, it is the image of an idempotent endomorphism of $(\bZ_q[G]b \otimes_{\bZ_q} \bZ_q[N_G(D)]b'^{\op})^{\oplus n}$, which \S\ref{subsec:proj} lifts to an idempotent endomorphism of $(\bS_q[G]b \otimes \bS_q[N_G(D)]b'^{\op})^{\oplus n}$. This idempotent splits in $\Bimod(\bS_q[G]b,\bS_q[N_G(D)]b')$, so $N_0'$ lifts to a projective bimodule as well. By Proposition \ref{prop:PM} it follows that the map $N'_0 \to N_0$ lifts to a map $N' \to N$, and $M:=\mathrm{Cone}(N' \to N)$ lifts $M_0$.
\end{proof}

\begin{thm}
\label{thm:last}
Let $b$ be a block of $\bS_q[G]$, let $D$ be its defect group, let $b'$ be the Brauer correspondent block of $\bS_q[N_G(D)]$. Then there is an equivalence of stable $\infty$-categories $\LMod(\bS_q[G]b) \cong \LMod(\bS_q[N_G(D)]b')$ that carries $\LModft{\bS_q[G]b}$ onto $\LModft{\bS_q[N_G(D)]b'}$ and $\LMod(\bS_q[G]b)^{\omega}$ onto $\LMod(\bS_q[N_G(D)]b')^{\omega}$. 
\end{thm}

\begin{proof}
This is a consequence of Proposition \ref{prop:conservative}. We apply it in the case where $k \to k'$ is the truncation map $\bS_q \to \bZ_q$, $A = \bS_q[N_G(D)]b'$, $B = \bS_q[G]b$, and $F(A)$ is the bimodule $M$ of Proposition \ref{prop:rouq-S_q}. By Rouquier's Theorem \ref{thm:rouquier}, $M \otimes_{\bS_q} \bZ_q$ induces an equivalence $\LMod(\bZ_q[N_G(D)]b') \cong \LMod(\bZ_q[G]b)$ --- the hypotheses of Proposition \ref{prop:conservative} are satisfied and so the functor $M \otimes_{\bS_q[N_G(D)]b'}-$ induces an equivalence $\LMod(\bS_q[N_G(D)]b') \cong \LMod(\bS_q[G]b)$. Like any equivalence, it preserves the subcategories of compact objects and so restricts to an equivalence $\LMod(\bS_q[N_G(D)]b')^{\omega}\cong \LMod(\bS_q[G]b)^{\omega} $.
Since each term of the complex $M$ of \Cref{prop:rouq-S_q} carries  $\LModft{\bS_q[N_G(D)]b'}$  to $\LModft{\bS_q[G]b}$, the equivalence carries $\LModft{\bS_q[N_G(D)]b'}$ into $\LModft{\bS_q[G]b}$. Finally, this functor $\LModft{\bS_q[N_G(D)]b'} \to \LModft{\bS_q[G]b}$ is an equivalence by Proposition \ref{prop: finiteness}.
\end{proof}

\bibliographystyle{alpha}
\bibliography{Bibliography}

\end{document}